\documentclass[12pt]{amsart}
\usepackage{mathrsfs,pstricks,ifpdf,tikz}
\usepackage{amsfonts}
\usepackage{enumerate}
\usepackage{makecell}
\newcommand{\bahao}{\fontsize{6pt}{\baselineskip}\selectfont}
\newcommand{\tabincell}[2]{\begin{tabular}{@{}#1@{}}#2\end{tabular}}
\usepackage{latexsym}
\usepackage{amsmath,amssymb}
\usepackage{amsthm}
\usepackage{ifthen,calc}
\usepackage{color}
\usepackage{graphicx}
\usepackage{latexsym}
\usepackage{multirow}
 \usepackage{diagbox}
 \usepackage{float}
\usepackage[format=hang, margin=10pt]{caption}

\newtheorem{theorem}{Theorem}[section]

\newtheorem{corollary}[theorem]{Corollary}

\newtheorem{lemma}[theorem]{Lemma}

\newcommand{\ssum} {\sum_{i=1}^{\infty} }
\def\dsum{\displaystyle\sum}
\def\rs{\rm\scriptsize}

\newcommand{\mdef}[1]{\textit{\textbf{#1}}}  
\newcommand{\vsa}{\vskip-12pt}
\newcommand{\vsb}{\vskip-6pt}
\newcommand{\Z}{\mathbb{Z}}

\newcounter{hours}
\newcounter{minutes}
\newcommand{\printtime}{
    \setcounter{hours}{\time/60}%
    \setcounter{minutes}{\time-\value{hours}*60}
    \ifthenelse{\value{hours}<10}{0}{}\thehours:%
    \ifthenelse{\value{minutes}<10}{0}{}\theminutes}

\begin{document}
\title{The thickness of the Kronecker product of graphs}

\author{Xia Guo}
\address{School of Mathematical Sciences, Xiamen University, 361005, Xiamen, China}
\email{guoxia@stu.xmu.edu.cn}
\author{Yan Yang}
\address{School of Mathematics, Tianjin University, 300350, Tianjin, China}
\email{yanyang@tju.edu.cn    (Corresponding author: Yan YANG)}
\thanks{This work was supported by NNSF of China under Grant  No. 11401430}

\begin{abstract}
The thickness of a graph $G$ is the minimum number of planar subgraphs whose union is $G$. In this paper, we present sharp lower and upper bounds for the thickness of the Kronecker product $G\times H$ of two graphs $G$ and $H$. We also give the exact thickness numbers for the Kronecker product
graphs $K_n\times K_2$, $K_{m,n}\times K_2$ and $K_{n,n,n}\times K_2$.
\end{abstract}


\maketitle

\section{Introduction}

\noindent
The {\it thickness} $\theta(G)$ of graph $G$ is the minimum number of planar subgraphs whose union is $G$. It is a measurement of the planarity of a graph, the graph with $\theta(G)=1$ is a planar graph. Since W.T.Tutte\cite{Tut63} inaugurated the thickness problem in 1963, the thickness of some classic types of graphs have been obtained by various authors, such as \cite{AG76,BH65,BHM64,Kl67,Vas76,Yan16} etc. In recent years, some authors focus on the thickness of the graphs which are obtained by operating on two graphs, such as the Cartesian product graph\cite{CYi16,YC17} and join graph\cite{CY17}. In this paper, we are concerned with the Kronecker product graph.

The {\it Kronecker product} (also called as tensor product, direct product, categorical product) $G\times H$ of graphs $G$ and $H$ is the graph   whose vertex set is  $V(G\times H)=V(G)\times V(H)$ and edge set is $E(G\times H)=\{(g,h)(g',h')|gg'\in E(G)~\mbox{and}~hh'\in E(H)\}.$ Figure \ref{figure 7} shows the Kronecker product graph $K_5\times K_2$ in which $\{u_1,\dots,u_5\}$ and $\{v_1,v_2\}$ are the vertex sets of the complete graph $K_5$ and $K_2$, respectively. Many authors did research on various topics of the Kronecker product graph, such as for its planarity\cite{BD09,FW77}, connectivity\cite{WY13}, colouring\cite{DS85,KS96} and application\cite{LC10} etc.

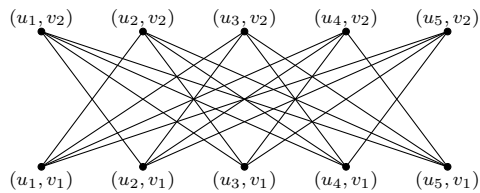
\begin{figure}[H]
\begin{center}
\begin{tikzpicture}
[scale=0.9]
\tikzstyle{every node}=[font=\tiny,scale=0.88]
\tiny
[inner sep=0pt]
\filldraw [black] (0,0) circle (1.4pt)
                  (1.5,0) circle (1.4pt)
                  (3,0) circle (1.4pt)
                  (4.5,0) circle (1.4pt)
                   (6,0) circle (1.4pt);
\filldraw [black] (0,2) circle (1.4pt)
                  (1.5,2) circle (1.4pt)
                  (3,2) circle (1.4pt)
                  (4.5,2) circle (1.4pt)
                   (6,2) circle (1.4pt);
\draw (0,0)--(1.5,2);\draw (0,0)--(3,2);\draw (0,0)--(4.5,2);\draw (0,0)--(6,2);
\draw (1.5,0)--(0,2);\draw (1.5,0)--(3,2);\draw (1.5,0)--(4.5,2);\draw (1.5,0)--(6,2);
\draw (3,0)--(0,2);\draw (3,0)--(1.5,2);\draw (3,0)--(4.5,2);\draw (3,0)--(6,2);
\draw (4.5,0)--(0,2);\draw (4.5,0)--(1.5,2);\draw (4.5,0)--(3,2);\draw (4.5,0)--(6,2);
\draw (6,0)--(0,2);\draw (6,0)--(1.5,2);\draw (6,0)--(3,2);\draw (6,0)--(4.5,2);
\draw (0,-0.2)  node { $(u_1,v_1)$};
\draw (1.5,-0.2)  node { $(u_2,v_1)$};
\draw (3,-0.2)  node { $(u_3,v_1)$};
\draw (4.5,-0.2)  node { $(u_4,v_1)$};
\draw (6,-0.2)  node { $(u_5,v_1)$};
\draw (0,2.2)  node { $(u_1,v_2)$};
\draw (1.5,2.2)  node { $(u_2,v_2)$};
\draw (3,2.2)  node { $(u_3,v_2)$};
\draw (4.5,2.2)  node { $(u_4,v_2)$};
\draw (6,2.2)  node { $(u_5,v_2)$};

\end{tikzpicture}
\caption{The Kronecker product graph $K_5\times K_2$}
\label{figure 7}
\end{center}
\end{figure}

The {\it complete graph} $K_n$ is the graph on $n$ vertices in which any two vertices are adjacent. The {\it complete bipartite graph} $K_{m,n}$ is the graph whose vertex set can be partitioned into two parts $X$ and $Y$, $|X|=m$ and $|Y|=n$, every edge has its ends in different parts and every two vertices in different parts are adjacent. The {\it complete tripartite graph} $K_{l,m,n}$ is defined analogously.

In this paper, we present lower and upper bounds for the thickness of the Kronecker product of two graphs in Section $2$, in which the lower bound comes from Euler's formula and the upper bound is derived from the structure of the Kronecker product graph. Then we study the thickness of the Kronecker product of graph with $K_{2}$. There are two reasons why we interested in it. One reason is that the upper bound for the thickness of the Kronecker product of two graphs we just given rely on that of the Kronecker product of graph with $K_{2}$. Another reason is that the planarity of the Kronecker product of two graphs have been characterized in \cite{FW77}, but graph with $K_{2}$ is one of its missing cases. It's a difficult case, because there exist non-planar graphs whose Kronecker product with $K_2$ are planar graphs, see Figures $1$ and $2$ in \cite{BD09} for example. In Sections $3$ and $4$, we provide the exact thickness numbers for the Kronecker product graphs $K_n\times K_2$, $K_{m,n}\times K_2$ and $K_{n,n,n}\times K_2$.

For undefined terminology, see \cite{BM08}.

\section{Thickness of the Kronecker product graph $G\times H$}

\begin{theorem}\label{2.1} Let $G$ and $H$ be two simple graphs on at least two vertices, then
\begin{equation}\notag
\Big\lceil\frac{2|E(G)| |E(H)|}{3|V(G)| |V(H)|-6}\Big\rceil  \leq
\theta(G\times H)\leq Min\{ \sum\limits_{\substack {i=1 \\  e_i\in E(H)}}^{|E(H)|}  \theta(G\times e_i),
\sum\limits_{\substack {j=1 \\  e_j\in E(G)}}^{|E(G)|}   \theta(H\times e_j)\}.
\end{equation}
\end{theorem}

\begin{proof}
It is easy to observe that the number of edges in $G\times H$ is $|E(G\times H)|=2|E(G)||E(H)|$ and the number of vertices in $G\times H$ is $|V(G\times H)|=|V(G)||V(H)|$. From the Euler's Formula, the planar graph with $|V(G)||V(H)|$ vertices, has at most $3|V(G)||V(H)|-6$ edges, the lower bound follows.

From the structure of the Kronecker product graph, we have
$$G\times H=\mathop{\cup}\limits_{\substack {i=1 \\  e_i\in E(H)}}^{|E(H)|} G\times e_i= \mathop{\cup}\limits_{\substack {j=1 \\  e_j\in E(G)}}^{|E(G)|} H\times e_j,$$
the upper bound can be derived easily.
\end{proof}

In the following, we will give examples to show both the lower and upper bound in Theorem \ref{2.1} are sharp. Let $G$ and $H$ be the graphs as shown in Figure \ref{figure 11}(a) and (b) respectively.  Figure \ref{figure 11}(c) illustrates a planar embedding of the graph $G\times \{v_1v_2\}$, in which
we denote the vertex $(u_i, v_j)$ by $u^j_i$, $1\leq i\leq 7$, $1\leq j\leq 2$. So the thickness of $G\times \{v_1v_2\}$ is one which meets the lower bound in Theorem \ref{2.1}. Figure \ref{figure 11}(d) illustrates a planar embedding of the graph $G\times \{v_2v_3\}$ which is isomorphic to $G\times \{v_1v_2\}$. Because $G\times H= G\times \{v_1v_2\}~\cup~G\times \{v_2v_3\}$, we get a planar subgraph decomposition of $G\times H$ with two subgraphs, which shows the thickness of $G\times H$ is not more than two. On the other hand, the graph $G\times H$ contains a subdivision of $K_5$ which are exhibited in Figure \ref{figure 11}(e), so $G\times H$ is not a planar graph, its thickness is greater than one. Therefore, the thickness of $G\times H$ is two which meet the upper bound in  Theorem \ref{2.1}.
\begin{figure}[H]
\begin{center}
\begin{tikzpicture}
[scale=0.85]
\tikzstyle{every node}=[font=\tiny,scale=0.8]
\tiny
[inner sep=0pt]
\filldraw [black] (0,0) circle (1.2pt)
                  (2,0) circle (1.2pt)
                  (-0.7,2) circle (1.2pt)
                  (2.7,2) circle (1.2pt)
                  (1,3.5) circle (1.2pt);
\filldraw [black] (0.1,2.7) circle (1.2pt)
                  (1.9,2.7) circle (1.2pt);
\draw (0,0)--(2,0)--(2.7,2)--(1,3.5)--(-0.7,2)--(0,0);
\draw (0,0)--(1,3.5)--(2,0)--(-0.7,2)--(2.7,2)--(0,0);
\draw (1.25,3.5) node {\tiny $u_1$};
\draw (2.95,2) node {\tiny $u_2$};
\draw (2.25,0) node {\tiny $u_3$};
\draw (-0.25,0) node {\tiny $u_4$};
\draw (-0.95,2) node {\tiny $u_5$};
\draw (2.15,2.7) node {\tiny $u_6$};
\draw (-0.15,2.7) node {\tiny $u_7$};
\begin{scope}[xshift=6cm]
\filldraw [black] (0,1.5) circle (1.2pt)
                  (1.5, 1.5) circle (1.2pt)
                  (3,1.5) circle (1.2pt);
\draw (0,1.5)--(3,1.5);
\draw (0,1.25) node {\tiny $v_1$};
\draw (1.5,1.25) node {\tiny $v_2$};
\draw (3,1.25) node {\tiny $v_3$};
\end{scope}
\end{tikzpicture}
\\ (a) The graph $G$\quad\quad\quad\quad\quad \quad\quad  (b) The graph $H$
\end{center}
\end{figure}

\begin{figure}[H]
\begin{center}
\begin{tikzpicture}
[scale=0.8]
\tikzstyle{every node}=[font=\tiny,scale=0.7]
\tiny
[inner sep=0pt]
\filldraw [black] (0,0) circle (1.2pt)
                  (1,0) circle (1.2pt)
                  (2,0) circle (1.2pt)
                  (3,0) circle (1.2pt)
                  (4,0) circle (1.2pt)
                  (5,0) circle (1.2pt)
                   (6,0) circle (1.2pt);
\filldraw [black] (4,1.5) circle (1.2pt)
                  (4,-1.5) circle (1.2pt);
\filldraw [black] (4,0.7) circle (1.2pt)
                  (4,-0.7) circle (1.2pt);
\filldraw [black] (4,2.5) circle (1.2pt)
                  (4,-2.5) circle (1.2pt);
\filldraw [black] (4,3) circle (1.2pt);
\draw (0,0)--(2,0);\draw (3,0)--(5,0);
\draw (4,1.5)--(2,0);\draw (4,1.5)--(3,0);
\draw (4,1.5)--(4,0);\draw (4,1.5)--(5,0);
\draw (4,2.5)--(0,0);\draw (4,2.5)--(2,0);
\draw (4,2.5)--(5,0);\draw (4,2.5)--(6,0);
\draw (4,-1.5)--(2,0);\draw (4,-1.5)--(3,0);
\draw (4,-1.5)--(4,0);\draw (4,-1.5)--(5,0);
\draw (4,-2.5)--(0,0);\draw (4,-2.5)--(2,0);
\draw (4,-2.5)--(5,0);\draw (4,-2.5)--(6,0);
\draw (0,0)--(4,3)--(6,0);
\draw (0,-0.25) node {\tiny $u_1^2$};
\draw (1,-0.25) node {\tiny $u_6^1$};
\draw (1.9,-0.25) node {\tiny $u_2^2$};
\draw (2.9,-0.25) node {\tiny $u_3^2$};
\draw (3.8,-0.2) node {\tiny $u_1^1$};
\draw (5.1,-0.25) node {\tiny $u_4^2$};
\draw (6.05,-0.25) node {\tiny $u_5^2$};
\draw (3.8,3.1) node {\tiny $u_7^1$};
\draw (3.8,2.63) node {\tiny $u_3^1$};
\draw (4,1.7) node {\tiny $u_5^1$};
\draw (3.8,0.7) node {\tiny $u_7^2$};
\draw (3.8,-2.63) node {\tiny $u_4^1$};
\draw (4,-1.7) node {\tiny $u_2^1$};
\draw (3.8,-0.7) node {\tiny $u_6^2$};
\begin{scope}[xshift=8cm]
\filldraw [black] (0,0) circle (1.2pt)
                  (1,0) circle (1.2pt)
                  (2,0) circle (1.2pt)
                  (3,0) circle (1.2pt)
                  (4,0) circle (1.2pt)
                  (5,0) circle (1.2pt)
                   (6,0) circle (1.2pt);
\filldraw [black] (4,1.5) circle (1.2pt)
                  (4,-1.5) circle (1.2pt);
\filldraw [black] (4,0.7) circle (1.2pt)
                  (4,-0.7) circle (1.2pt);
\filldraw [black] (4,2.5) circle (1.2pt)
                  (4,-2.5) circle (1.2pt);
\filldraw [black] (4,3) circle (1.2pt);
\draw (0,0)--(2,0);\draw (3,0)--(5,0);
\draw (4,1.5)--(2,0);\draw (4,1.5)--(3,0);
\draw (4,1.5)--(4,0);\draw (4,1.5)--(5,0);
\draw (4,2.5)--(0,0);\draw (4,2.5)--(2,0);
\draw (4,2.5)--(5,0);\draw (4,2.5)--(6,0);
\draw (4,-1.5)--(2,0);\draw (4,-1.5)--(3,0);
\draw (4,-1.5)--(4,0);\draw (4,-1.5)--(5,0);
\draw (4,-2.5)--(0,0);\draw (4,-2.5)--(2,0);
\draw (4,-2.5)--(5,0);\draw (4,-2.5)--(6,0);
\draw (0,0)--(4,3)--(6,0);
\draw (0,-0.25) node {\tiny $u_1^2$};
\draw (1,-0.25) node {\tiny $u_6^3$};
\draw (1.9,-0.25) node {\tiny $u_2^2$};
\draw (2.9,-0.25) node {\tiny $u_3^2$};
\draw (3.8,-0.2) node {\tiny $u_1^3$};
\draw (5.1,-0.25) node {\tiny $u_4^2$};
\draw (6.05,-0.25) node {\tiny $u_5^2$};
\draw (3.8,3.1) node {\tiny $u_7^3$};
\draw (3.8,2.63) node {\tiny $u_3^3$};
\draw (4,1.7) node {\tiny $u_5^3$};
\draw (3.8,0.7) node {\tiny $u_7^2$};
\draw (3.8,-2.63) node {\tiny $u_4^3$};
\draw (4,-1.7) node {\tiny $u_2^3$};
\draw (3.8,-0.7) node {\tiny $u_6^2$};
\end{scope}
\end{tikzpicture}
\\ (c) The graph $G\times \{v_1v_2\}$\quad \quad\quad\quad\quad\quad
 (d) The graph $G\times \{v_2v_3\}$
\end{center}
\end{figure}

\begin{figure}[H]
\begin{center}
\begin{tikzpicture}
[scale=1]
\tikzstyle{every node}=[font=\tiny,scale=0.85]
\tiny
[inner sep=0pt]
\filldraw [black] (0,0) circle (1.2pt)
                  (2,0) circle (1.2pt)
                  (-0.7,2) circle (1.2pt)
                  (2.7,2) circle (1.2pt)
                  (1,3.5) circle (1.2pt);
\filldraw [black] (0.1,2.7) circle (1.2pt)
                  (1.9,2.7) circle (1.2pt)
                  (-0.4,1.1) circle (1.2pt)
                  (2.4,1.1) circle (1.2pt)
                  (1,0) circle (1.2pt);
\filldraw [black] (0.45,1.6) circle (1.2pt)
                  (1.55,1.6) circle (1.2pt)
                  (0.65,0.99) circle (1.2pt)
                  (1.35,0.99) circle (1.2pt)
                  (1,2) circle (1.2pt);
\draw (0,0)--(2,0)--(2.7,2)--(1,3.5)--(-0.7,2)--(0,0);
\draw (0,0)--(1,3.5)--(2,0)--(-0.7,2)--(2.7,2)--(0,0);
\draw (1.25,3.5) node {\tiny $u_1^1$};
\draw (2.95,2) node {\tiny $u_2^1$};
\draw (2.25,0) node {\tiny $u_3^1$};
\draw (-0.25,0) node {\tiny $u_4^1$};
\draw (-0.95,2) node {\tiny $u_5^1$};
\draw (2.15,2.75) node {\tiny $u_6^2$};
\draw (-0.15,2.75) node {\tiny $u_7^2$};
\draw (-0.6,1.1) node {\tiny $u_2^2$};
\draw (2.6,1.1) node {\tiny $u_5^2$};
\draw (1,-0.25) node {\tiny $u_1^2$};
\draw (1,2.2) node {\tiny $u_3^2$};
\draw (0.25,1.6) node {\tiny $u_2^3$};
\draw (1.75,1.6) node {\tiny $u_4^3$};
\draw (0.6,0.75) node {\tiny $u_4^2$};
\draw (1.45,0.83) node {\tiny $u_3^3$};
\end{tikzpicture}
\\ (e) A subgraph of $G\times H$
\caption{An example to show both lower and upper bounds in Theorem $2.1$ are sharp}
\label{figure 11}
\end{center}
\end{figure}
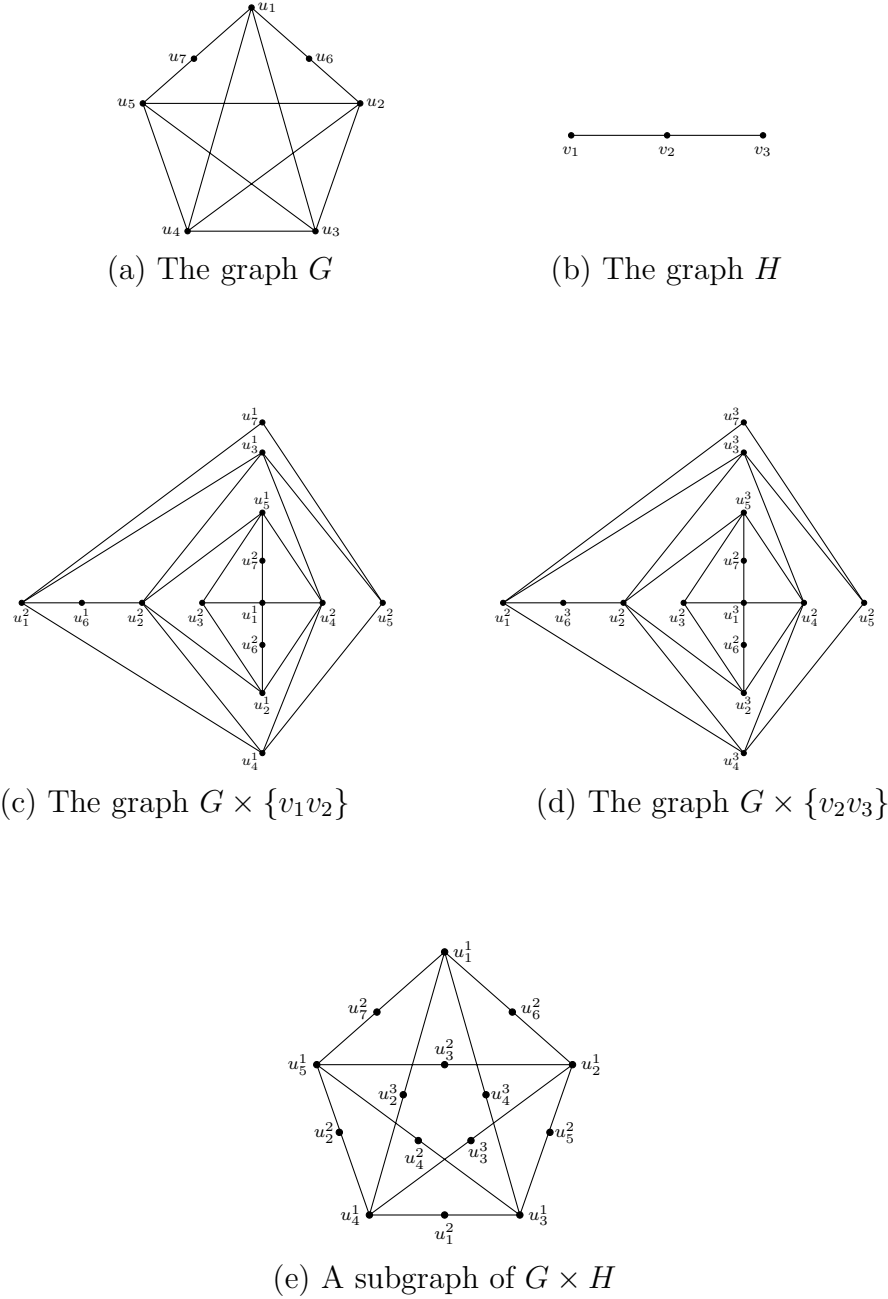

If $G\times H$ does not contain any triangle, from the Euler's Formula, the planar graph with $|V(G)||V(H)|$ vertices, has at most $2|V(G)||V(H)|-4$ edges, a tighter lower bound can be derived.

\begin{theorem}\label{2.2} Let $G$ and $H$ be two simple graphs on at least two vertices. If $G\times H$ does not contain any triangle, then
\begin{equation}\notag
\Big\lceil\frac{|E(G)| |E(H)|}{|V(G)| |V(H)|-2}\Big\rceil  \leq
\theta(G\times H)\leq Min\{ \sum\limits_{\substack {i=1 \\  e_i\in E(H)}}^{|E(H)|}  \theta(G\times e_i),
\sum\limits_{\substack {j=1 \\  e_j\in E(G)}}^{|E(G)|}   \theta(H\times e_j)\}.
\end{equation}
\end{theorem}

\section{The thickness of $K_n\times K_2$ and $K_{m,n}\times K_2$ }

Let $G$ be a simple graph with $n$ vertices, $V(G)=\{v_1,\ldots,v_n\}$ and $V(K_2)=\{1,2\}$. Then $G\times K_2$ is a bipartite graph, the two vertex parts are $\{(v_i,1)|1\leq i\leq n\}$ and $\{(v_i,2)|1\leq i\leq n\}$, so $G\times K_2$ is a subgraph of $K_{n,n}$ which shows that $\theta(G\times K_2)\leq\theta(K_{n,n})$.

Although the thickness of the complete bipartite $K_{m,n}$ have not been solved completely, when $m=n$, we have the following result.

\begin{lemma}\cite{BHM64}\label{le2} The thickness of the complete bipartite graph $K_{n,n}$ is
$$\theta(K_{n,n})=\big\lceil\frac{n+2}{4}\big\rceil.$$ \end{lemma}

When $n=4p ~(p \geq 1)$, Chen and Yin gave a planar subgraphs decomposition of $K_{4p,4p}$ with $p+1$ planar subgraphs $G_1,\ldots, G_{p+1}$ in \cite{CYi16}. Denote the two vertex parts of $K_{4p,4p}$ by $U=\{u_1,\ldots, u_{4p}\}$ and $V=\{v_1,\ldots, v_{4p}\}$, Figure \ref{figure 2}  shows their planar subgraphs decomposition of $K_{4p,4p}$, in which for each $G_{r} (1 \leq r\leq p)$, both $v_{4r-3}$ and $v_{4r-1}$ join to each vertex in set $\bigcup\limits^p_{i=1,i\neq r}\{u_{4i-3},u_{4i-2}\}$, both $v_{4r-2}$ and $v_{4r}$ join to each vertex in set $\bigcup\limits^p_{i=1,i\neq r}\{u_{4i-1},u_{4i}\}$,
both $u_{4r-1}$ and $u_{4r}$ join to each vertex in set $\bigcup\limits^p_{i=1,i\neq r}\{v_{4i-3},v_{4i-1}\}$, and both $u_{4r-3}$ and $u_{4r-2}$
join to each vertex in set $\bigcup\limits^p_{i=1,i\neq r}\{v_{4i-2},v_{4i}\}$. Notice that $G_{p+1}$ is a perfect matching of $K_{4p,4p}$, the edge set of it is $\{u_iv_i|1\leq i\leq 4p\}$.

\begin{figure}[H]
\begin{center}
\begin{tikzpicture}                                                         
[inner sep=0pt]
\filldraw [black] (0,2.1) circle (1.4pt)
                  (2.1,0) circle (1.4pt)
                  (2.1,4.2) circle (1.4pt)
                  (4.2,2.1) circle (1.4pt);
\filldraw [black]     (1.4,2.1) circle (1.4pt)
                      (2.8,2.1) circle (1.4pt)
                      (2.1,1.4) circle (1.4pt)
                      (2.1,2.8) circle (1.4pt);
\filldraw [black] (0.7,2.5) circle (1.4pt)
                  (0.7,1.7) circle (1.4pt);
\filldraw [black]      (3.5,2.5) circle (1.4pt)
                       (3.5,1.7) circle (1.4pt);
\filldraw [black] (2.5,0.7) circle (1.4pt)
                  (1.7,0.7) circle (1.4pt);
\filldraw [black]      (2.5,3.5) circle (1.4pt)
                       (1.7,3.5) circle (1.4pt);
\filldraw [black] (0.7,2.2) circle (0.7pt)
                  (0.7,2.05) circle (0.7pt)
                  (0.7,1.9) circle (0.7pt);
\filldraw [black]      (3.5,2.2) circle (0.7pt)
                       (3.5,2.05) circle (0.7pt)
                       (3.5,1.9) circle (0.7pt);
\filldraw [black] (2.2,0.7) circle (0.7pt)
                  (2.05,0.7) circle (0.7pt)
                  (1.9,0.7) circle (0.7pt);
\filldraw [black] (2.2,3.5) circle (0.7pt)
                  (2.05,3.5) circle (0.7pt)
                  (1.9,3.5) circle (0.7pt);
\draw (0,2.1)--(2.1,0);\draw (2.1,0)--(4.2,2.1);\draw (2.1,4.2)--(0,2.1);\draw ((4.2,2.1)--(2.1,4.2);
\draw (1.4,2.1)--(2.1,1.4);\draw (2.1,1.4)--(2.8,2.1);\draw (2.8,2.1)--(2.1,2.8);\draw (1.4,2.1)--(2.1,2.8);
\draw (0,2.1)--(0.7,2.5);\draw (0,2.1)--(0.7,1.7);
\draw (1.4,2.1)--(0.7,2.5);\draw (1.4,2.1)--(0.7,1.7);
\draw (2.8,2.1)--(3.5,2.5);\draw (2.8,2.1)--(3.5,1.7) ;
\draw (4.2,2.1)--(3.5,2.5);\draw (4.2,2.1)--(3.5,1.7) ;
\draw (2.1,0)--(2.5,0.7);\draw (2.1,0)--(1.7,0.7);
\draw (2.1,1.4)--(2.5,0.7);\draw (2.1,1.4)--(1.7,0.7);
\draw (2.1,2.8)--(2.5,3.5);\draw (2.1,2.8)--(1.7,3.5);
\draw (2.1,4.2)--(2.5,3.5);\draw (2.1,4.2)--(1.7,3.5);
\draw (2.1,-0.2) node {\tiny $v_{4r}$}
      (2.1,4.4) node {\tiny $v_{4r-3}$}
      (-0.26,2.3) node {\tiny $u_{4r-1}$}
      (4.5,2.25) node {\tiny $u_{4r-2}$};
\draw    (2.55,1.35) node {\tiny $v_{4r-2}$}
          (1.68,2.83) node {\tiny $v_{4r-1}$}
           (1.4,1.85) node {\tiny $u_{4r}$}
             (2.35,2.1) node {\tiny $u_{4r-3}$};
\draw[-](2.1,4.2)..controls+(-1.2,-1.2)and+(-0.1,0.5)..(1.4,2.1);
\draw[-](0,2.1)..controls+(1.2,-1.2)and+(-0.5,-0.1)..(2.1,1.4);
\draw[-](2.1,0)..controls+(1.2,1.2)and+(0.1,-0.5)..(2.8,2.1);
\draw[-](4.2,2.1)..controls+(-1.2,1.2)and+(0.5,0.1)..(2.1,2.8);
\node(5) at(2.1,2.45)[circle,draw]{\scriptsize $1$};
\node(4) at(2.8,2.6)[circle,draw]{\scriptsize $2$};
\node(2) at(1,0.3)[circle,draw]{\scriptsize $3$};
\draw[-](0.7,2) to (-0.5,1.4);\draw (-1.2,1.2) node {\tiny $\bigcup\limits^p_{i=1,i\neq r}\{v_{4i-3},v_{4i-1}\}$};
\draw[-](3.6,2.1) to (4.2,3.2);\draw (4.8,3.5) node {\tiny $\bigcup\limits^p_{i=1,i\neq r}\{v_{4i-2},v_{4i}\}$};
\draw[-](2.00,0.8) to (3.5,1);\draw (5.00,1.0) node {\tiny $\bigcup\limits^p_{i=1,i\neq r}\{u_{4i-1},u_{4i}\}$};
\draw[-](2.05,3.5) to (0.7,3.9);\draw (-0.2,4.03) node {\tiny $\bigcup\limits^p_{i=1,i\neq r}\{u_{4i-3},u_{4i-2}\}$};
\end{tikzpicture}
\\ (a) The graph $G_r\ \ (1\leq r\leq p)$
\end{center}
\end{figure}
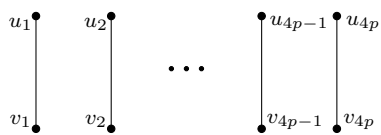
\begin{figure}[H]
\begin{center}
\begin{tikzpicture}
[xscale=1]
\tikzstyle{every node}=[font=\tiny,scale=1]
[inner sep=0pt]\tiny

\filldraw [black] (0,1) circle (1.4pt)
                   (0,-0.5) circle (1.4pt)
                  (1,1) circle (1.4pt)
                   (1,-0.5) circle (1.4pt)
                   (3,1) circle (1.4pt)
                   (3,-0.5) circle (1.4pt)
                  (4,1) circle (1.4pt)
                  (4,-0.5) circle (1.4pt);
\draw(4,1)-- (4,-0.5);\draw(3,1)-- (3,-0.5);
\draw(1,1)-- (1,-0.5);\draw(0,1)-- (0,-0.5);
\filldraw [black] (1.8,0.3) circle (0.8pt)
                               (2,0.3) circle (0.8pt)
                               (2.2,0.3) circle (0.8pt);
\draw (-0.2,0.9) node {\tiny $u_{1}$}
      (-0.2,-0.4) node {\tiny $v_{1}$}
      (0.8,0.9) node {\tiny $u_{2}$}
      (0.8,-0.4) node {\tiny $v_{2}$};
\draw (3. 5,0.9) node {\tiny $u_{4p-1}$}
      (3.45,-0.4) node {\tiny $v_{4p-1}$}
      (4.35,0.9) node {\tiny $u_{4p}$}
      (4.3,-0.4) node {\tiny $v_{4p}$};
\end{tikzpicture}
\\ (b) The graph $G_{p+1}$
\caption{A planar decomposition of $K_{4p,4p}$}
\label{figure 2}
\end{center}
\end{figure}

\begin{lemma}\cite{CYi16}\label{l3.2} Suppose $K_{n,n}$ is a complete bipartite graph with two vertex parts $U=\{u_1,\ldots, u_{n}\}$ and $V=\{v_1,\ldots, v_{n}\}$. When $n=4p$, there exists a planar subgraphs decomposition of $K_{4p,4p}$ with $p+1$ planar subgraphs $G_1,\ldots, G_{p+1}$ in which $G_{p+1}$ is a perfect matching of $K_{4p,4p}$ with edge set $\{u_iv_i|1\leq i\leq 4p\}$.\end{lemma}

\begin{theorem}\label{th1} The thickness of the Kronecker product of $K_n$ and $K_2$ is
$$\theta(K_{n}\times K_2)=\big\lceil\frac{n}{4}\big\rceil.$$ \end{theorem}

\begin{proof}
Suppose that the vertex sets of $K_n$ and $K_2$ are $\{x_1,\dots,x_n\}$ and $\{1,2\}$ respectively. The graph $K_n\times K_2$ is a bipartite graph whose two vertex parts are $\{(x_i,1)|1\leq i\leq n\}$ and $\{(x_i,2)|1\leq i\leq n\}$, and edge set is $\{(x_i,1)(x_j,2)|1\leq i,j\leq n, i\neq j \}$. For $1\leq i\leq n$, $1\leq k\leq 2$, we denote the vertex $(x_i,k)$ of $K_n\times K_2$ by $x_i^k$ for simplicity.

Since $|E(K_{n}\times K_2)|=n(n-1)$ and $|V(K_{n}\times K_2)|=2n$, from Theorem \ref{2.2}, we have
$$\theta(K_n\times K_2)\geq \big\lceil\frac{n(n-1)}{4n-4}\big\rceil=\big\lceil\frac{n}{4}\big\rceil.\eqno(1)$$

In the following, we will construct planar decompositions of $K_{n}\times K_2$ with $\big\lceil\frac{n}{4}\big\rceil$ subgraphs to complete the proof.

{\bf Case 1.} When $n=4p$.
\\Suppose that $K_{n,n}$ is a complete bipartite graph with vertex partition  $(X^1,X^2)$ in which $X^1=\{x_1^1,\dots,x_n^1\}$ and $X^2=\{x_1^2,\dots,x_n^2\}$. The graph $G_{p+1}$ is a perfect matching of $K_{4p,4p}$ whose edge set is $\{x_{i}^1x_{i}^2|1\leq i\leq n\}$, then
$K_{n}\times K_2=K_{n,n}-G_{p+1}$. From Lemma \ref{l3.2}, there exists a planar decomposition $\{G_1,\ldots, G_{p}\}$ of $K_n\times K_2$ in which $G_{r}(1\leq r\leq p)$ is isomorphic to the graph in Figure \ref{figure 2}(a). Therefore, $\theta(K_{4p}\times K_2)\leq p$.

{\bf Case 2.} When $n=4p+2$.
\\When $p\geq 1$, we draw a graph $G'_{p+1}$ as shown in Figure \ref{figure 1}, then
$\{G_1,\ldots, G_{p}, G'_{p+1}\}$ is a planar decomposition of $K_{4p+2}\times K_2$ with $p+1$ subgraphs, so we have $\theta(K_{4p+2}\times K_2)\leq p+1$. When $n=2$, $K_2\times K_2=2K_2$ is a planar graph.

\begin{figure}[H]
\begin{center}
\begin{tikzpicture}
[xscale=1]
\tikzstyle{every node}=[font=\tiny,scale=0.8]
[inner sep=0pt]\tiny
\filldraw [black] (2,0) circle (0.8pt)
                  (2.3,0) circle (0.8pt)
                  (1.7,0) circle (0.8pt);
\filldraw [black] (7,0) circle (0.8pt)
                  (7.3,0) circle (0.8pt)
                  (6.7,0) circle (0.8pt);
\filldraw [black] (0,0) circle (1.4pt)
                   (1,0) circle (1.4pt)
                  (3,0) circle (1.4pt)
                   (4,0) circle (1.4pt)
                   (5,0) circle (1.4pt)
                   (6,0) circle (1.4pt)
                  (8,0) circle (1.4pt)
                  (9,0) circle (1.4pt);
\filldraw [black] (2,1) circle (1.4pt)
                   (2,-1) circle (1.4pt)
                  (7,1) circle (1.4pt)
                   (7,-1) circle (1.4pt);
\draw(2,1)--(0,0);\draw(2,1)--(1,0);\draw(2,1)--(3,0);\draw(2,1)--(4,0);
\draw(2,-1)--(0,0);\draw(2,-1)--(1,0);\draw(2,-1)--(3,0);\draw(2,-1)--(4,0); \draw(7,1)--(5,0);\draw(7,1)--(6,0);\draw(7,1)--(8,0);\draw(7,1)--(9,0);
\draw(7,-1)--(5,0);\draw(7,-1)--(6,0);\draw(7,-1)--(8,0);\draw(7,-1)--(9,0);
\draw(2,1)--(7,1);\draw(2,-1)--(7,-1);
\draw (2,1.2) node {\tiny $x_{4p+1}^1$}
      (7.25,1.2) node {\tiny $x_{4p+2}^2$}
      (2,-1.2) node {\tiny $x_{4p+2}^1$}
      (7.25,-1.2) node {\tiny $x_{4p+1}^2$};
\draw (0.3,0) node {\tiny $x_{1}^2$}
      (1.3,0) node {\tiny $x_{2}^2$}
      (3.4,0) node {\tiny $x_{4p-1}^2$}
      (4.3,0) node {\tiny $x_{4p}^2$};
\draw (5.3,0) node {\tiny $x_{1}^1$}
      (6.3,0) node {\tiny $x_{2}^1$}
      (8.4,0) node {\tiny $x_{4p-1}^1$}
      (9.3,0) node {\tiny $x_{4p}^1$};
\end{tikzpicture}
\caption{The graph $G'_{p+1}$}
\label{figure 1}
\end{center}
\end{figure}
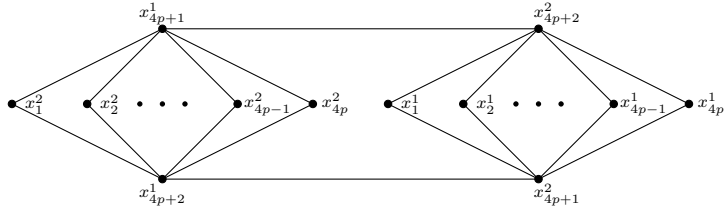

{\bf Case 3.} When $n=4p+1$ and $n=4p+3$.
\\Because $K_{4p+1}\times K_2$ is a subgraph of  $K_{4p+2}\times K_2$, we have $\theta(K_{4p+1}\times K_2)\leq \theta(K_{4p+2}\times K_2)=p+1$. Similarly, when $n=4p+3$, we have $\theta(K_{4p+3}\times K_2)\leq \theta(K_{4(p+1)}\times K_2)=p+1$.
\\Summarizing Cases $1$, $2$ and $3$, we have  $$\theta(K_n\times K_2)\leq \big\lceil\frac{n}{4}\big\rceil.\eqno(2)$$

Theorem follows from inequalities $(1)$ and $(2)$.
\end{proof}

\begin{theorem}\label{th3.4} Let $G$ be a simple graph on $n (n\geq 2)$ vertices, then \begin{equation}\notag
\big\lceil\frac{E(G)}{2n-2}\big\rceil  \leq
\theta(G\times K_2)\leq \big\lceil\frac{n}{4}\big\rceil.
\end{equation} \end{theorem}

\begin{proof} Because $G\times K_2$ is a subgraph of $K_{n}\times K_2$, we have $\theta(G\times K_2)\leq \theta(K_{n}\times K_2)$. Combining it with Theorems \ref{2.2} and \ref{th1}, the theorem follows.
\end{proof}

\begin{lemma}\cite{FW77}\label{le1} $K_{m,n}\times K_{p,q}=K_{mp,nq}\cup K_{mq,np}.$\end{lemma}

\begin{theorem}\label{t} The thickness of the Kronecker product of $K_{m,n}$ and $K_{p,q}$ is
$$\theta(K_{m,n}\times K_{p,q})=Max\{\theta(K_{mp,nq}), \theta(K_{mq,np})\}.$$ \end{theorem}
\begin{proof} From Lemma \ref{le1}, the proof is straightforward.
\end{proof}

Because $K_2$ is also $K_{1,1}$, the following corollaries are easy to get, from Theorem \ref{t} and Lemma \ref{le2}.

\begin{corollary}\label{Cor3} The thickness of the Kronecker product of $K_{m,n}$ and $K_{2}$ is
$$\theta(K_{m,n}\times K_{2})=\theta(K_{m,n}).$$ \end{corollary}

\begin{corollary}\label{Cor2} The thickness of the Kronecker product of $K_{n,n}$ and $K_2$ is
$$\theta(K_{n,n}\times K_2)=\big\lceil\frac{n+2}{4}\big\rceil.$$ \end{corollary}

\section{The thickness of the Kronecker product graph $K_{n,n,n}\times K_2$ }

Let $(X,Y,Z)$ be the vertex partition of the complete tripartite graph $K_{l,m,n}~(l\leq m\leq n)$ in which $X=\{x_1,\dots,x_l\}$, $Y=\{y_1,\dots,y_m\}$, $Z=\{z_1,\dots,z_n\}$. Let $\{1,2\}$ be the vertex set of $K_2$.  We denote the vertex $(v,k)$ of $K_{l,m,n}\times K_2$ by $v^k$ in which $v\in V(K_{n,n,n})$ and $k\in \{1,2\}$. For $k=1,2$, we denote $X^k=\{x_1^k,\dots,x_l^k\}$, $Y^k=\{y_1^k,\dots,y_m^k\}$ and $Z^k=\{z_1^k,\dots,z_n^k\}$. In Figure \ref{figure 3}, we draw a sketch of the graph $K_{l,m,n}\times K_2$, in which the edge joining two vertex set indicates that each vertex in one vertex set  is adjacent to each vertex in another vertex set. Suppose $G(X^1,Y^2)$ is the graph induced by the vertex sets $X^1$ and $Y^2$ of $K_{l,m,n}\times K_2$, then $G(X^1,Y^2)$ is  isomorphic to $K_{l,m}$, the graphs $G(Y^1,Z^2)$, $G(Z^1,X^2)$, $G(X^2,Y^1)$, $G(Y^2,Z^1)$ and $G(Z^2,X^1)$ are defined analogously. We define $$G^1=G(X^1,Y^2)\cup G(Y^1,Z^2)\cup G(Z^1,X^2)$$and $$G^2=G(X^2,Y^1)\cup G(Y^2,Z^1)\cup G(Z^2,X^1),$$ then $K_{l,m,n}\times K_2=G^1\cup G^2$.

\begin{figure}[H]
\begin{center}
\begin{tikzpicture}
[xscale=1]
\tikzstyle{every node}=[font=\tiny,scale=0.8]
[inner sep=0pt]\tiny
\node(11) at(0,2)[circle,draw]{\tiny $X^1$};
\node(12) at(2,2)[circle,draw]{\tiny $Y^1$};
\node(13) at(4,2)[circle,draw]{\tiny $Z^1$};
\node(21) at(0,0)[circle,draw]{\tiny $X^2$};
\node(22) at(2,0)[circle,draw]{\tiny $Y^2$};
\node(23) at(4,0)[circle,draw]{\tiny $Z^2$};
\draw[-](11) to  (22);\draw[-](11) to  (23);
\draw[-](12) to  (21);\draw[-](12) to  (23);
\draw[-](13) to  (21);\draw[-](13) to  (22);
\end{tikzpicture}
\caption{The graph $K_{l,m,n}\times K_2$}
\label{figure 3}
\end{center}
\end{figure}
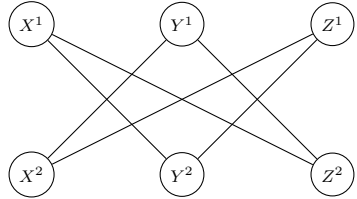

\begin{theorem}\label{th3} The thickness of the Kronecker product graph $K_{l,m,n}\times K_2~(l\leq m\leq n)$ satisfies the inequality
$$
\big\lceil\frac{lm+ln+mn}{2(l+m+n)-2}\big\rceil    \leq  \theta(K_{l,m,n}\times K_2)   \leq     2\theta(K_{m,n}).
$$
\end{theorem}
\begin{proof}
From Theorem \ref{th3.4}, one can get the lower bound in this theorem easily.
Any two graphs of $G(X^1,Y^2)$, $G(Y^1,Z^2)$ and $G(Z^1,X^2)$ are disjoint with each other and $l\leq m\leq n$, so we have $$\theta(G^1)\leq Max\{\theta(G(X^1,Y^2),\theta(G(Y^1,Z^2),\theta(G(Z^1,X^2)\}=\theta(K_{m,n}).$$ Similarly, we have $$\theta(G^2)\leq Max\{\theta(G(X^2,Y^1),\theta(G(Y^2,Z^1),\theta(G(Z^2,X^1))\}=\theta(K_{m,n}).$$
Due to the graph $K_{l,m,n}\times K_2=G^1 \cup G^2$, we have $\theta(K_{l,m,n}\times K_2)\leq 2\theta(K_{m,n})$. Summarizing the above, the theorem is obtained.
\end{proof}

In the following, we will discuss the thickness of $K_{n,n,n}\times K_2$ and  we will see when $n=4p+2$, the upper and lower bound in Theorem \ref{th3} are equal, so both bounds in Theorem \ref{th3} are sharp.

\begin{lemma}\label{le3} When $n=4p$, there exists a planar decomposition of the Kronecker product graph $K_{n,n,n}\times K_2$ with $2p+1$ subgraphs.
\end{lemma}
\begin{proof}
Because $|X^k|=|Y^k|=|Z^k|=n$ $(k=1,2)$, all the graphs $G(X^1,Y^2)$, \\$ G(Y^1,Z^2), G(Z^1,X^2), G(X^2,Y^1), G(Y^2,Z^1), G(Z^2,X^1)$ are isomorphic to $K_{n,n}$.

Let $\{G_1,\dots,G_{p+1}\}$ be the planar decomposition of $K_{n,n}$ as shown in Figure \ref{figure 2}. For $1\leq r\leq p+1$, $G_r$ is a bipartite graph, so we also denote it by $G_r(V,U)$. In $G_r(V,U)$, we replace the vertex set $V$ by $X^1$, $U$ by $Y^2$, i.e.,
for each $1\leq i\leq n$, replace the vertex $v_i$ by $x_i^1$, and $u_i$ by $y_i^2$, then we get graph $ G_r(X^1,Y^2)$. Analogously, we  obtain graphs $G_r(Y^1,Z^2)$,$G_r(Z^1,X^2)$, $G_r(X^2,Y^1)$, $G_r(Y^2,Z^1)$ and $G_r(Z^2,X^1)$.

For $1\leq r\leq p+1$, let $$G^1_r=G_r(X^1,Y^2)\cup G_r(Y^1,Z^2)\cup G_r(Z^1,X^2)$$
and $$G^2_{r}=G_{r}(X^2,Y^1)\cup G_r(Y^2,Z^1)\cup G_r(Z^2,X^1).$$
Because $G_r(X^1,Y^2),G_r(Y^1,Z^2),G_r(Z^1,X^2)$ are all planar graphs and they are disjoint with each other, $G^1_r$ is a planar graph.
For the same reason, we have that $G^2_{r}$ is also a planar graph.

Let graph $G_{p+1}$ be the graph $G^1_{p+1}\cup G^2_{p+1}$.  We have
\begin{align}
G_{p+1} &=G^1_{p+1}\cup G^2_{p+1}\notag\\
               &= \{\mathop{\cup}\limits^{n}_{i=1}(x^1_iy^2_i\cup y^1_iz^2_i\cup z^1_ix^2_i)\}   \cup
               \{\mathop{\cup}\limits^{n}_{i=1}(x^2_iy^1_i\cup y^2_iz^1_i\cup z^2_ix^1_i)\}\notag\\
               &=\mathop{\cup}\limits^{n}_{i=1}(x^1_iy^2_iz^1_ix^2_iy^1_iz^2_ix^1_i).\notag
\end{align}
It is easy to see $G_{p+1}$ consists of $n$ disjoint cycles of length $6$, hence $G_{p+1}$ is a planar graph.

Because
$$G(X^1,Y^2)=\mathop{\cup}\limits^{p+1}_{r=1} G_r(X^1,Y^2),~~~~~~~~~~ G(Y^1,Z^2)=\mathop{\cup}\limits^{p+1}_{r=1} G_r(Y^1,Z^2), $$ $$G(Z^1,X^2)=\mathop{\cup}\limits^{p+1}_{r=1} G_r(Z^1,X^2),~~~~~~~~~~ G(X^2,Y^1)=\mathop{\cup}\limits^{p+1}_{r=1} G_r(X^2,Y^1),$$
and $$G(Y^2,Z^1)=\mathop{\cup}\limits^{p+1}_{r=1} G_r(Y^2,Z^1),~~~~~~~~~~ G(Z^2,X^1)=\mathop{\cup}\limits^{p+1}_{r=1} G_r(Z^2,X^1),$$
we have \begin{align}
K_{n,n,n}\times K_2 &=G^1\cup G^2\notag\\
               &= \mathop{\cup}\limits^{p+1}_{r=1}(G_r^1\cup G_r^2)\notag\\
               &=\mathop{\cup}\limits^{p}_{r=1}(G_r^1\cup G_r^2)\cup G_{p+1}.\notag
\end{align}
\\So we get a planar decomposition  of $K_{4p,4p,4p}\times K_2$ with $2p+1$ subgraphs $G^1_1$, $\dots$, $G^1_p$, $G^2_1$, $\dots$, $G^2_p$, $G_{p+1}$. The proof is completed.
\end{proof}

 We draw the planar decomposition of  $K_{8,8,8}\times K_2$ as shown in Figure \ref{figure 8}.
\begin{figure}[H]
\begin{center}
\begin{tikzpicture}
[scale=0.7]
\tikzstyle{every node}=[font=\tiny,scale=0.7]
\tiny
[inner sep=0pt]
\filldraw [black] (1.5,0) circle (1.2pt)
                  (-0.5,0) circle (1.2pt)
                  (0.5,1) circle (1.2pt)
                  (0.5,-1) circle (1.2pt);
\filldraw [black] (3.5,0) circle (1.2pt)
                  (-2.5,0) circle (1.2pt)
                  (0.5,3) circle (1.2pt)
                  (0.5,-3) circle (1.2pt);
\filldraw [black] (2.5,0.15) circle (1.2pt)
                  (2.5,-0.15) circle (1.2pt)
                  (-1.5,0.15) circle (1.2pt)
                  (-1.5,-0.15) circle (1.2pt);
\filldraw [black] (0.65,2) circle (1.2pt)
                  (0.65,-2) circle (1.2pt)
                  (0.35,2) circle (1.2pt)
                  (0.35,-2) circle (1.2pt);
\draw  (1.5,0)--(0.5,1)--(-0.5,0)--(0.5,-1)--(1.5,0);
\draw  (3.5,0)--(0.5,3)--(-2.5,0)--(0.5,-3)--(3.5,0);
\draw  (1.5,0)--(2.5,0.15)--(3.5,0)--(2.5,-0.15)--(1.5,0);
\draw  (-0.5,0)--(-1.5,0.15)--(-2.5,0)--(-1.5,-0.15)--(-0.5,0);
\draw  (0.5,-1)--(0.65,-2)--(0.5,-3)--(0.35,-2)--(0.5,-1);
\draw  (0.5,1)--(0.65,2)--(0.5,3)--(0.35,2)--(0.5,1);
\draw  (0.5,-3)..controls+(0.1,0.1)and+(0.5,-1.2)..(1.5,0);
\draw  (0.5,3)..controls+(-0.01,-0.01)and+(-0.5,1.2)..(-0.5,0);
\draw  (-2.5,0)..controls+(0.1,-0.1)and+(-1.2,-0.5)..(0.5,-1);
\draw  (3.5,0)..controls+(-0.1,0.1)and+(1.2,0.5)..(0.5,1);
\draw  (2.5,0.4)  node { $x^1_6$};\draw  (2.45,-0.40)  node { $x^1_8$};
\draw  (-1.5,0.44)  node { $x^1_5$};\draw  (-1.5,-0.40)  node { $x^1_7$};
\draw  (0.25,1.1)  node { $x^1_3$};\draw  (0.75,-1)  node { $x^1_2$};
\draw  (0.24,3.1)  node { $x^1_1$};\draw  (0.24,-3.1)  node { $x^1_4$};
\draw  (0.9,2)  node { $y^2_6$};\draw  (0.10,2)  node { $y^2_5$};
\draw  (0.9,-2)  node { $y^2_8$};\draw  (0.10,-2)  node { $y^2_7$};
\draw  (1.6,0.26)  node { $y^2_1$};\draw  (-0.6,-0.25)  node { $y^2_4$};
\draw  (3.6,-0.25)  node { $y^2_2$};\draw  (-2.55,0.26)  node { $y^2_3$};
\filldraw [black] (8,0) circle (1.2pt)
                  (6,0) circle (1.2pt)
                  (7,1) circle (1.2pt)
                  (7,-1) circle (1.2pt);
\filldraw [black] (10,0) circle (1.2pt)
                  (4,0) circle (1.2pt)
                  (7,3) circle (1.2pt)
                  (7,-3) circle (1.2pt);
\filldraw [black] (9,0.15) circle (1.2pt)
                  (9,-0.15) circle (1.2pt)
                  (5,0.15) circle (1.2pt)
                  (5,-0.15) circle (1.2pt);
\filldraw [black] (7.15,2) circle (1.2pt)
                  (7.15,-2) circle (1.2pt)
                  (6.85,2) circle (1.2pt)
                  (6.85,-2) circle (1.2pt);
\draw  (8,0)--(7,1)--(6,0)--(7,-1)--(8,0);
\draw  (10,0)--(7,3)--(4,0)--(7,-3)--(10,0);
\draw  (8,0)--(9,0.15)--(10,0)--(9,-0.15)--(8,0);
\draw  (6,0)--(5,0.15)--(4,0)--(5,-0.15)--(6,0);
\draw  (7,-1)--(7.15,-2)--(7,-3)--(6.85,-2)--(7,-1);
\draw  (7,1)--(7.15,2)--(7,3)--(6.85,2)--(7,1);
\draw  (7,-3)..controls+(0.1,0.1)and+(0.5,-1.2)..(8,0);
\draw  (7,3)..controls+(-0.01,-0.01)and+(-0.5,1.2)..(6,0);
\draw  (4,0)..controls+(0.1,-0.1)and+(-1.2,-0.5)..(7,-1);
\draw  (10,0)..controls+(-0.1,0.1)and+(1.2,0.5)..(7,1);
\draw  (9,0.4)  node { $y^1_6$};\draw  (8.95,-0.40)  node { $y^1_8$};
\draw  (5,0.44)  node { $y^1_5$};\draw  (5,-0.40)  node { $y^1_7$};
\draw  (6.75,1.1)  node { $y^1_3$};\draw  (7.25,-1)  node { $y^1_2$};
\draw  (6.74,3.1)  node { $y^1_1$};\draw  (6.74,-3.1)  node { $y^1_4$};
\draw  (7.4,2)  node { $z^2_6$};\draw  (6.60,2)  node { $z^2_5$};
\draw  (7.4,-2)  node { $z^2_8$};\draw  (6.60,-2)  node { $z^2_7$};
\draw  (8.1,0.26)  node { $z^2_1$};\draw  (5.9,-0.25)  node { $z^2_4$};
\draw  (10.1,-0.25)  node { $z^2_2$};\draw  (3.95,0.26)  node { $z^2_3$};
\filldraw [black] (14.5,0) circle (1.2pt)
                  (12.5,0) circle (1.2pt)
                  (13.5,1) circle (1.2pt)
                  (13.5,-1) circle (1.2pt);
\filldraw [black] (16.5,0) circle (1.2pt)
                  (10.5,0) circle (1.2pt)
                  (13.5,3) circle (1.2pt)
                  (13.5,-3) circle (1.2pt);
\filldraw [black] (15.5,0.15) circle (1.2pt)
                  (15.5,-0.15) circle (1.2pt)
                  (11.5,0.15) circle (1.2pt)
                  (11.5,-0.15) circle (1.2pt);
\filldraw [black] (13.65,2) circle (1.2pt)
                  (13.65,-2) circle (1.2pt)
                  (13.35,2) circle (1.2pt)
                  (13.35,-2) circle (1.2pt);
\draw  (14.5,0)--(13.5,1)--(12.5,0)--(13.5,-1)--(14.5,0);
\draw  (16.5,0)--(13.5,3)--(10.5,0)--(13.5,-3)--(16.5,0);
\draw  (14.5,0)--(15.5,0.15)--(16.5,0)--(15.5,-0.15)--(14.5,0);
\draw  (12.5,0)--(11.5,0.15)--(10.5,0)--(11.5,-0.15)--(12.5,0);
\draw  (13.5,-1)--(13.65,-2)--(13.5,-3)--(13.35,-2)--(13.5,-1);
\draw  (13.5,1)--(13.65,2)--(13.5,3)--(13.35,2)--(13.5,1);
\draw  (13.5,-3)..controls+(0.1,0.1)and+(0.5,-1.2)..(14.5,0);
\draw  (13.5,3)..controls+(-0.01,-0.01)and+(-0.5,1.2)..(12.5,0);
\draw  (10.5,0)..controls+(0.1,-0.1)and+(-1.2,-0.5)..(13.5,-1);
\draw  (16.5,0)..controls+(-0.1,0.1)and+(1.2,0.5)..(13.5,1);
\draw  (15.5,0.4)  node { $z^1_6$};\draw  (15.45,-0.40)  node { $z^1_8$};
\draw  (11.5,0.44)  node { $z^1_5$};\draw  (11.5,-0.40)  node { $z^1_7$};
\draw  (13.25,1.1)  node { $z^1_3$};\draw  (13.75,-1)  node { $z^1_2$};
\draw  (13.24,3.1)  node { $z^1_1$};\draw  (13.24,-3.1)  node { $z^1_4$};
\draw  (13.9,2)  node { $x^2_6$};\draw  (13.10,2)  node { $x^2_5$};
\draw  (13.9,-2)  node { $x^2_8$};\draw  (13.10,-2)  node { $x^2_7$};
\draw  (14.6,0.26)  node { $x^2_1$};\draw  (12.4,-0.25)  node { $x^2_4$};
\draw  (16.5,-0.27)  node { $x^2_2$};\draw  (10.45,0.26)  node { $x^2_3$};
\end{tikzpicture}
\\ (a) The graph $G^1_1$
\end{center}
\end{figure}

\begin{figure}[H]
\begin{center}
\begin{tikzpicture}
[scale=0.7]
\tikzstyle{every node}=[font=\tiny,scale=0.7]
\tiny
[inner sep=0pt]
\filldraw [black] (1.5,0) circle (1.2pt)
                  (-0.5,0) circle (1.2pt)
                  (0.5,1) circle (1.2pt)
                  (0.5,-1) circle (1.2pt);
\filldraw [black] (3.5,0) circle (1.2pt)
                  (-2.5,0) circle (1.2pt)
                  (0.5,3) circle (1.2pt)
                  (0.5,-3) circle (1.2pt);
\filldraw [black] (2.5,0.15) circle (1.2pt)
                  (2.5,-0.15) circle (1.2pt)
                  (-1.5,0.15) circle (1.2pt)
                  (-1.5,-0.15) circle (1.2pt);
\filldraw [black] (0.65,2) circle (1.2pt)
                  (0.65,-2) circle (1.2pt)
                  (0.35,2) circle (1.2pt)
                  (0.35,-2) circle (1.2pt);
\draw  (1.5,0)--(0.5,1)--(-0.5,0)--(0.5,-1)--(1.5,0);
\draw  (3.5,0)--(0.5,3)--(-2.5,0)--(0.5,-3)--(3.5,0);
\draw  (1.5,0)--(2.5,0.15)--(3.5,0)--(2.5,-0.15)--(1.5,0);
\draw  (-0.5,0)--(-1.5,0.15)--(-2.5,0)--(-1.5,-0.15)--(-0.5,0);
\draw  (0.5,-1)--(0.65,-2)--(0.5,-3)--(0.35,-2)--(0.5,-1);
\draw  (0.5,1)--(0.65,2)--(0.5,3)--(0.35,2)--(0.5,1);
\draw  (0.5,-3)..controls+(0.1,0.1)and+(0.5,-1.2)..(1.5,0);
\draw  (0.5,3)..controls+(-0.01,-0.01)and+(-0.5,1.2)..(-0.5,0);
\draw  (-2.5,0)..controls+(0.1,-0.1)and+(-1.2,-0.5)..(0.5,-1);
\draw  (3.5,0)..controls+(-0.1,0.1)and+(1.2,0.5)..(0.5,1);
\draw  (2.5,0.4)  node { $x^1_2$};\draw  (2.45,-0.40)  node { $x^1_4$};
\draw  (-1.5,0.44)  node { $x^1_1$};\draw  (-1.5,-0.40)  node { $x^1_3$};
\draw  (0.25,1.1)  node { $x^1_7$};\draw  (0.75,-1)  node { $x^1_6$};
\draw  (0.24,3.1)  node { $x^1_5$};\draw  (0.24,-3.1)  node { $x^1_8$};
\draw  (0.9,2)  node { $y^2_2$};\draw  (0.10,2)  node { $y^2_1$};
\draw  (0.9,-2)  node { $y^2_4$};\draw  (0.10,-2)  node { $y^2_3$};
\draw  (1.6,0.26)  node { $y^2_5$};\draw  (-0.6,-0.25)  node { $y^2_8$};
\draw  (3.6,-0.25)  node { $y^2_6$};\draw  (-2.55,0.26)  node { $y^2_7$};
\filldraw [black] (8,0) circle (1.2pt)
                  (6,0) circle (1.2pt)
                  (7,1) circle (1.2pt)
                  (7,-1) circle (1.2pt);
\filldraw [black] (10,0) circle (1.2pt)
                  (4,0) circle (1.2pt)
                  (7,3) circle (1.2pt)
                  (7,-3) circle (1.2pt);
\filldraw [black] (9,0.15) circle (1.2pt)
                  (9,-0.15) circle (1.2pt)
                  (5,0.15) circle (1.2pt)
                  (5,-0.15) circle (1.2pt);
\filldraw [black] (7.15,2) circle (1.2pt)
                  (7.15,-2) circle (1.2pt)
                  (6.85,2) circle (1.2pt)
                  (6.85,-2) circle (1.2pt);
\draw  (8,0)--(7,1)--(6,0)--(7,-1)--(8,0);
\draw  (10,0)--(7,3)--(4,0)--(7,-3)--(10,0);
\draw  (8,0)--(9,0.15)--(10,0)--(9,-0.15)--(8,0);
\draw  (6,0)--(5,0.15)--(4,0)--(5,-0.15)--(6,0);
\draw  (7,-1)--(7.15,-2)--(7,-3)--(6.85,-2)--(7,-1);
\draw  (7,1)--(7.15,2)--(7,3)--(6.85,2)--(7,1);
\draw  (7,-3)..controls+(0.1,0.1)and+(0.5,-1.2)..(8,0);
\draw  (7,3)..controls+(-0.01,-0.01)and+(-0.5,1.2)..(6,0);
\draw  (4,0)..controls+(0.1,-0.1)and+(-1.2,-0.5)..(7,-1);
\draw  (10,0)..controls+(-0.1,0.1)and+(1.2,0.5)..(7,1);
\draw  (9,0.4)  node { $y^1_2$};\draw  (8.95,-0.40)  node { $y^1_4$};
\draw  (5,0.44)  node { $y^1_1$};\draw  (5,-0.40)  node { $y^1_3$};
\draw  (6.75,1.1)  node { $y^1_7$};\draw  (7.25,-1)  node { $y^1_6$};
\draw  (6.74,3.1)  node { $y^1_5$};\draw  (6.74,-3.1)  node { $y^1_8$};
\draw  (7.4,2)  node { $z^2_2$};\draw  (6.60,2)  node { $z^2_1$};
\draw  (7.4,-2)  node { $z^2_4$};\draw  (6.60,-2)  node { $z^2_3$};
\draw  (8.1,0.26)  node { $z^2_5$};\draw  (5.9,-0.25)  node { $z^2_8$};
\draw  (10.1,-0.25)  node { $z^2_6$};\draw  (3.95,0.26)  node { $z^2_7$};
\filldraw [black] (14.5,0) circle (1.2pt)
                  (12.5,0) circle (1.2pt)
                  (13.5,1) circle (1.2pt)
                  (13.5,-1) circle (1.2pt);
\filldraw [black] (16.5,0) circle (1.2pt)
                  (10.5,0) circle (1.2pt)
                  (13.5,3) circle (1.2pt)
                  (13.5,-3) circle (1.2pt);
\filldraw [black] (15.5,0.15) circle (1.2pt)
                  (15.5,-0.15) circle (1.2pt)
                  (11.5,0.15) circle (1.2pt)
                  (11.5,-0.15) circle (1.2pt);
\filldraw [black] (13.65,2) circle (1.2pt)
                  (13.65,-2) circle (1.2pt)
                  (13.35,2) circle (1.2pt)
                  (13.35,-2) circle (1.2pt);
\draw  (14.5,0)--(13.5,1)--(12.5,0)--(13.5,-1)--(14.5,0);
\draw  (16.5,0)--(13.5,3)--(10.5,0)--(13.5,-3)--(16.5,0);
\draw  (14.5,0)--(15.5,0.15)--(16.5,0)--(15.5,-0.15)--(14.5,0);
\draw  (12.5,0)--(11.5,0.15)--(10.5,0)--(11.5,-0.15)--(12.5,0);
\draw  (13.5,-1)--(13.65,-2)--(13.5,-3)--(13.35,-2)--(13.5,-1);
\draw  (13.5,1)--(13.65,2)--(13.5,3)--(13.35,2)--(13.5,1);
\draw  (13.5,-3)..controls+(0.1,0.1)and+(0.5,-1.2)..(14.5,0);
\draw  (13.5,3)..controls+(-0.01,-0.01)and+(-0.5,1.2)..(12.5,0);
\draw  (10.5,0)..controls+(0.1,-0.1)and+(-1.2,-0.5)..(13.5,-1);
\draw  (16.5,0)..controls+(-0.1,0.1)and+(1.2,0.5)..(13.5,1);
\draw  (15.5,0.4)  node { $z^1_2$};\draw  (15.45,-0.40)  node { $z^1_4$};
\draw  (11.5,0.44)  node { $z^1_1$};\draw  (11.5,-0.40)  node { $z^1_3$};
\draw  (13.25,1.1)  node { $z^1_7$};\draw  (13.75,-1)  node { $z^1_6$};
\draw  (13.24,3.1)  node { $z^1_5$};\draw  (13.24,-3.1)  node { $z^1_8$};
\draw  (13.9,2)  node { $x^2_2$};\draw  (13.10,2)  node { $x^2_1$};
\draw  (13.9,-2)  node { $x^2_4$};\draw  (13.10,-2)  node { $x^2_3$};
\draw  (14.6,0.26)  node { $x^2_5$};\draw  (12.4,-0.25)  node { $x^2_8$};
\draw  (16.5,-0.27)  node { $x^2_6$};\draw  (10.45,0.26)  node { $x^2_7$};
\end{tikzpicture}
\\ (b) The graph $G^1_2$
\end{center}
\end{figure}

\begin{figure}[H]
\begin{center}
\begin{tikzpicture}
[scale=0.7]
\tikzstyle{every node}=[font=\tiny,scale=0.7]
\tiny
[inner sep=0pt]
\filldraw [black] (1.5,0) circle (1.2pt)
                  (-0.5,0) circle (1.2pt)
                  (0.5,1) circle (1.2pt)
                  (0.5,-1) circle (1.2pt);
\filldraw [black] (3.5,0) circle (1.2pt)
                  (-2.5,0) circle (1.2pt)
                  (0.5,3) circle (1.2pt)
                  (0.5,-3) circle (1.2pt);
\filldraw [black] (2.5,0.15) circle (1.2pt)
                  (2.5,-0.15) circle (1.2pt)
                  (-1.5,0.15) circle (1.2pt)
                  (-1.5,-0.15) circle (1.2pt);
\filldraw [black] (0.65,2) circle (1.2pt)
                  (0.65,-2) circle (1.2pt)
                  (0.35,2) circle (1.2pt)
                  (0.35,-2) circle (1.2pt);
\draw  (1.5,0)--(0.5,1)--(-0.5,0)--(0.5,-1)--(1.5,0);
\draw  (3.5,0)--(0.5,3)--(-2.5,0)--(0.5,-3)--(3.5,0);
\draw  (1.5,0)--(2.5,0.15)--(3.5,0)--(2.5,-0.15)--(1.5,0);
\draw  (-0.5,0)--(-1.5,0.15)--(-2.5,0)--(-1.5,-0.15)--(-0.5,0);
\draw  (0.5,-1)--(0.65,-2)--(0.5,-3)--(0.35,-2)--(0.5,-1);
\draw  (0.5,1)--(0.65,2)--(0.5,3)--(0.35,2)--(0.5,1);
\draw  (0.5,-3)..controls+(0.1,0.1)and+(0.5,-1.2)..(1.5,0);
\draw  (0.5,3)..controls+(-0.01,-0.01)and+(-0.5,1.2)..(-0.5,0);
\draw  (-2.5,0)..controls+(0.1,-0.1)and+(-1.2,-0.5)..(0.5,-1);
\draw  (3.5,0)..controls+(-0.1,0.1)and+(1.2,0.5)..(0.5,1);
\draw  (2.5,0.4)  node { $x^2_6$};\draw  (2.45,-0.40)  node { $x^2_8$};
\draw  (-1.5,0.44)  node { $x^2_5$};\draw  (-1.5,-0.40)  node { $x^2_7$};
\draw  (0.25,1.1)  node { $x^2_3$};\draw  (0.75,-1)  node { $x^2_2$};
\draw  (0.24,3.1)  node { $x^2_1$};\draw  (0.24,-3.1)  node { $x^2_4$};
\draw  (0.9,2)  node { $y^1_6$};\draw  (0.10,2)  node { $y^1_5$};
\draw  (0.9,-2)  node { $y^1_8$};\draw  (0.10,-2)  node { $y^1_7$};
\draw  (1.6,0.26)  node { $y^1_1$};\draw  (-0.6,-0.25)  node { $y^1_4$};
\draw  (3.6,-0.25)  node { $y^1_2$};\draw  (-2.55,0.26)  node { $y^1_3$};
\filldraw [black] (8,0) circle (1.2pt)
                  (6,0) circle (1.2pt)
                  (7,1) circle (1.2pt)
                  (7,-1) circle (1.2pt);
\filldraw [black] (10,0) circle (1.2pt)
                  (4,0) circle (1.2pt)
                  (7,3) circle (1.2pt)
                  (7,-3) circle (1.2pt);
\filldraw [black] (9,0.15) circle (1.2pt)
                  (9,-0.15) circle (1.2pt)
                  (5,0.15) circle (1.2pt)
                  (5,-0.15) circle (1.2pt);
\filldraw [black] (7.15,2) circle (1.2pt)
                  (7.15,-2) circle (1.2pt)
                  (6.85,2) circle (1.2pt)
                  (6.85,-2) circle (1.2pt);
\draw  (8,0)--(7,1)--(6,0)--(7,-1)--(8,0);
\draw  (10,0)--(7,3)--(4,0)--(7,-3)--(10,0);
\draw  (8,0)--(9,0.15)--(10,0)--(9,-0.15)--(8,0);
\draw  (6,0)--(5,0.15)--(4,0)--(5,-0.15)--(6,0);
\draw  (7,-1)--(7.15,-2)--(7,-3)--(6.85,-2)--(7,-1);
\draw  (7,1)--(7.15,2)--(7,3)--(6.85,2)--(7,1);
\draw  (7,-3)..controls+(0.1,0.1)and+(0.5,-1.2)..(8,0);
\draw  (7,3)..controls+(-0.01,-0.01)and+(-0.5,1.2)..(6,0);
\draw  (4,0)..controls+(0.1,-0.1)and+(-1.2,-0.5)..(7,-1);
\draw  (10,0)..controls+(-0.1,0.1)and+(1.2,0.5)..(7,1);
\draw  (9,0.4)  node { $y^2_6$};\draw  (8.95,-0.40)  node { $y^2_8$};
\draw  (5,0.44)  node { $y^2_5$};\draw  (5,-0.40)  node { $y^2_7$};
\draw  (6.75,1.1)  node { $y^2_3$};\draw  (7.25,-1)  node { $y^2_2$};
\draw  (6.74,3.1)  node { $y^2_1$};\draw  (6.74,-3.1)  node { $y^2_4$};
\draw  (7.4,2)  node { $z^1_6$};\draw  (6.60,2)  node { $z^1_5$};
\draw  (7.4,-2)  node { $z^1_8$};\draw  (6.60,-2)  node { $z^1_7$};
\draw  (8.1,0.26)  node { $z^1_1$};\draw  (5.9,-0.25)  node { $z^1_4$};
\draw  (10.1,-0.25)  node { $z^1_2$};\draw  (3.95,0.26)  node { $z^1_3$};
\filldraw [black] (14.5,0) circle (1.2pt)
                  (12.5,0) circle (1.2pt)
                  (13.5,1) circle (1.2pt)
                  (13.5,-1) circle (1.2pt);
\filldraw [black] (16.5,0) circle (1.2pt)
                  (10.5,0) circle (1.2pt)
                  (13.5,3) circle (1.2pt)
                  (13.5,-3) circle (1.2pt);
\filldraw [black] (15.5,0.15) circle (1.2pt)
                  (15.5,-0.15) circle (1.2pt)
                  (11.5,0.15) circle (1.2pt)
                  (11.5,-0.15) circle (1.2pt);
\filldraw [black] (13.65,2) circle (1.2pt)
                  (13.65,-2) circle (1.2pt)
                  (13.35,2) circle (1.2pt)
                  (13.35,-2) circle (1.2pt);
\draw  (14.5,0)--(13.5,1)--(12.5,0)--(13.5,-1)--(14.5,0);
\draw  (16.5,0)--(13.5,3)--(10.5,0)--(13.5,-3)--(16.5,0);
\draw  (14.5,0)--(15.5,0.15)--(16.5,0)--(15.5,-0.15)--(14.5,0);
\draw  (12.5,0)--(11.5,0.15)--(10.5,0)--(11.5,-0.15)--(12.5,0);
\draw  (13.5,-1)--(13.65,-2)--(13.5,-3)--(13.35,-2)--(13.5,-1);
\draw  (13.5,1)--(13.65,2)--(13.5,3)--(13.35,2)--(13.5,1);
\draw  (13.5,-3)..controls+(0.1,0.1)and+(0.5,-1.2)..(14.5,0);
\draw  (13.5,3)..controls+(-0.01,-0.01)and+(-0.5,1.2)..(12.5,0);
\draw  (10.5,0)..controls+(0.1,-0.1)and+(-1.2,-0.5)..(13.5,-1);
\draw  (16.5,0)..controls+(-0.1,0.1)and+(1.2,0.5)..(13.5,1);
\draw  (15.5,0.4)  node { $z^2_6$};\draw  (15.45,-0.40)  node { $z^2_8$};
\draw  (11.5,0.44)  node { $z^2_5$};\draw  (11.5,-0.40)  node { $z^2_7$};
\draw  (13.25,1.1)  node { $z^2_3$};\draw  (13.75,-1)  node { $z^2_2$};
\draw  (13.24,3.1)  node { $z^2_1$};\draw  (13.24,-3.1)  node { $z^2_4$};
\draw  (13.9,2)  node { $x^1_6$};\draw  (13.10,2)  node { $x^1_5$};
\draw  (13.9,-2)  node { $x^1_8$};\draw  (13.10,-2)  node { $x^1_7$};
\draw  (14.6,0.26)  node { $x^1_1$};\draw  (12.4,-0.25)  node { $x^1_4$};
\draw  (16.5,-0.27)  node { $x^1_2$};\draw  (10.45,0.26)  node { $x^1_3$};
\end{tikzpicture}
\\ (c) The graph $G^2_1$
\end{center}
\end{figure}

\begin{figure}[H]
\begin{center}
\begin{tikzpicture}
[scale=0.7]
\tikzstyle{every node}=[font=\tiny,scale=0.7]
\tiny
[inner sep=0pt]
\filldraw [black] (1.5,0) circle (1.2pt)
                  (-0.5,0) circle (1.2pt)
                  (0.5,1) circle (1.2pt)
                  (0.5,-1) circle (1.2pt);
\filldraw [black] (3.5,0) circle (1.2pt)
                  (-2.5,0) circle (1.2pt)
                  (0.5,3) circle (1.2pt)
                  (0.5,-3) circle (1.2pt);
\filldraw [black] (2.5,0.15) circle (1.2pt)
                  (2.5,-0.15) circle (1.2pt)
                  (-1.5,0.15) circle (1.2pt)
                  (-1.5,-0.15) circle (1.2pt);
\filldraw [black] (0.65,2) circle (1.2pt)
                  (0.65,-2) circle (1.2pt)
                  (0.35,2) circle (1.2pt)
                  (0.35,-2) circle (1.2pt);
\draw  (1.5,0)--(0.5,1)--(-0.5,0)--(0.5,-1)--(1.5,0);
\draw  (3.5,0)--(0.5,3)--(-2.5,0)--(0.5,-3)--(3.5,0);
\draw  (1.5,0)--(2.5,0.15)--(3.5,0)--(2.5,-0.15)--(1.5,0);
\draw  (-0.5,0)--(-1.5,0.15)--(-2.5,0)--(-1.5,-0.15)--(-0.5,0);
\draw  (0.5,-1)--(0.65,-2)--(0.5,-3)--(0.35,-2)--(0.5,-1);
\draw  (0.5,1)--(0.65,2)--(0.5,3)--(0.35,2)--(0.5,1);
\draw  (0.5,-3)..controls+(0.1,0.1)and+(0.5,-1.2)..(1.5,0);
\draw  (0.5,3)..controls+(-0.01,-0.01)and+(-0.5,1.2)..(-0.5,0);
\draw  (-2.5,0)..controls+(0.1,-0.1)and+(-1.2,-0.5)..(0.5,-1);
\draw  (3.5,0)..controls+(-0.1,0.1)and+(1.2,0.5)..(0.5,1);
\draw  (2.5,0.4)  node { $x^2_2$};\draw  (2.45,-0.40)  node { $x^2_4$};
\draw  (-1.5,0.44)  node { $x^2_1$};\draw  (-1.5,-0.40)  node { $x^2_3$};
\draw  (0.25,1.1)  node { $x^2_7$};\draw  (0.75,-1)  node { $x^2_6$};
\draw  (0.24,3.1)  node { $x^2_5$};\draw  (0.24,-3.1)  node { $x^2_8$};
\draw  (0.9,2)  node { $y^1_2$};\draw  (0.10,2)  node { $y^1_1$};
\draw  (0.9,-2)  node { $y^1_4$};\draw  (0.10,-2)  node { $y^1_3$};
\draw  (1.6,0.26)  node { $y^1_5$};\draw  (-0.6,-0.25)  node { $y^1_8$};
\draw  (3.6,-0.25)  node { $y^1_6$};\draw  (-2.55,0.26)  node { $y^1_7$};
\filldraw [black] (8,0) circle (1.2pt)
                  (6,0) circle (1.2pt)
                  (7,1) circle (1.2pt)
                  (7,-1) circle (1.2pt);
\filldraw [black] (10,0) circle (1.2pt)
                  (4,0) circle (1.2pt)
                  (7,3) circle (1.2pt)
                  (7,-3) circle (1.2pt);
\filldraw [black] (9,0.15) circle (1.2pt)
                  (9,-0.15) circle (1.2pt)
                  (5,0.15) circle (1.2pt)
                  (5,-0.15) circle (1.2pt);
\filldraw [black] (7.15,2) circle (1.2pt)
                  (7.15,-2) circle (1.2pt)
                  (6.85,2) circle (1.2pt)
                  (6.85,-2) circle (1.2pt);
\draw  (8,0)--(7,1)--(6,0)--(7,-1)--(8,0);
\draw  (10,0)--(7,3)--(4,0)--(7,-3)--(10,0);
\draw  (8,0)--(9,0.15)--(10,0)--(9,-0.15)--(8,0);
\draw  (6,0)--(5,0.15)--(4,0)--(5,-0.15)--(6,0);
\draw  (7,-1)--(7.15,-2)--(7,-3)--(6.85,-2)--(7,-1);
\draw  (7,1)--(7.15,2)--(7,3)--(6.85,2)--(7,1);
\draw  (7,-3)..controls+(0.1,0.1)and+(0.5,-1.2)..(8,0);
\draw  (7,3)..controls+(-0.01,-0.01)and+(-0.5,1.2)..(6,0);
\draw  (4,0)..controls+(0.1,-0.1)and+(-1.2,-0.5)..(7,-1);
\draw  (10,0)..controls+(-0.1,0.1)and+(1.2,0.5)..(7,1);
\draw  (9,0.4)  node { $y^2_2$};\draw  (8.95,-0.40)  node { $y^2_4$};
\draw  (5,0.44)  node { $y^2_1$};\draw  (5,-0.40)  node { $y^2_3$};
\draw  (6.75,1.1)  node { $y^2_7$};\draw  (7.25,-1)  node { $y^2_6$};
\draw  (6.74,3.1)  node { $y^2_5$};\draw  (6.74,-3.1)  node { $y^2_8$};
\draw  (7.4,2)  node { $z^1_2$};\draw  (6.60,2)  node { $z^1_1$};
\draw  (7.4,-2)  node { $z^1_4$};\draw  (6.60,-2)  node { $z^1_3$};
\draw  (8.1,0.26)  node { $z^1_5$};\draw  (5.9,-0.25)  node { $z^1_8$};
\draw  (10.1,-0.25)  node { $z^1_6$};\draw  (3.95,0.26)  node { $z^1_7$};
\filldraw [black] (14.5,0) circle (1.2pt)
                  (12.5,0) circle (1.2pt)
                  (13.5,1) circle (1.2pt)
                  (13.5,-1) circle (1.2pt);
\filldraw [black] (16.5,0) circle (1.2pt)
                  (10.5,0) circle (1.2pt)
                  (13.5,3) circle (1.2pt)
                  (13.5,-3) circle (1.2pt);
\filldraw [black] (15.5,0.15) circle (1.2pt)
                  (15.5,-0.15) circle (1.2pt)
                  (11.5,0.15) circle (1.2pt)
                  (11.5,-0.15) circle (1.2pt);
\filldraw [black] (13.65,2) circle (1.2pt)
                  (13.65,-2) circle (1.2pt)
                  (13.35,2) circle (1.2pt)
                  (13.35,-2) circle (1.2pt);
\draw  (14.5,0)--(13.5,1)--(12.5,0)--(13.5,-1)--(14.5,0);
\draw  (16.5,0)--(13.5,3)--(10.5,0)--(13.5,-3)--(16.5,0);
\draw  (14.5,0)--(15.5,0.15)--(16.5,0)--(15.5,-0.15)--(14.5,0);
\draw  (12.5,0)--(11.5,0.15)--(10.5,0)--(11.5,-0.15)--(12.5,0);
\draw  (13.5,-1)--(13.65,-2)--(13.5,-3)--(13.35,-2)--(13.5,-1);
\draw  (13.5,1)--(13.65,2)--(13.5,3)--(13.35,2)--(13.5,1);
\draw  (13.5,-3)..controls+(0.1,0.1)and+(0.5,-1.2)..(14.5,0);
\draw  (13.5,3)..controls+(-0.01,-0.01)and+(-0.5,1.2)..(12.5,0);
\draw  (10.5,0)..controls+(0.1,-0.1)and+(-1.2,-0.5)..(13.5,-1);
\draw  (16.5,0)..controls+(-0.1,0.1)and+(1.2,0.5)..(13.5,1);
\draw  (15.5,0.4)  node { $z^2_2$};\draw  (15.45,-0.40)  node { $z^2_4$};
\draw  (11.5,0.44)  node { $z^2_1$};\draw  (11.5,-0.40)  node { $z^2_3$};
\draw  (13.25,1.1)  node { $z^2_7$};\draw  (13.75,-1)  node { $z^2_6$};
\draw  (13.24,3.1)  node { $z^2_5$};\draw  (13.24,-3.1)  node { $z^2_8$};
\draw  (13.9,2)  node { $x^1_2$};\draw  (13.10,2)  node { $x^1_1$};
\draw  (13.9,-2)  node { $x^1_4$};\draw  (13.10,-2)  node { $x^1_3$};
\draw  (14.6,0.26)  node { $x^1_5$};\draw  (12.4,-0.25)  node { $x^1_8$};
\draw  (16.5,-0.27)  node { $x^1_6$};\draw  (10.45,0.26)  node { $x^1_7$};
\end{tikzpicture}
\\ (d) The graph $G^2_2$
\end{center}
\end{figure}

\begin{figure}[H]
\begin{center}
\begin{tikzpicture}
[scale=0.7]
\tikzstyle{every node}=[font=\tiny,scale=0.7]
\tiny
[inner sep=0pt]
\filldraw [black] (0,0) circle (1.2pt)
                  (0,1.5) circle (1.2pt)
                  (0,3) circle (1.2pt);
\filldraw [black] (1.5,0) circle (1.2pt)
                  (1.5,1.5) circle (1.2pt)
                  (1.5,3) circle (1.2pt);
\draw  (0,0)--(1.5,0)--(1.5,3)--(0,3)--(0,0);
\draw  (0.25,0.25)  node { $x^1_1$};\draw  (1.3,0.25)  node { $z^2_1$};
\draw  (0.25,1.5)  node { $y^2_1$};\draw  (1.3,1.5)  node { $y^1_1$};
\draw  (0.25,2.75)  node { $z^1_1$};\draw  (1.3,2.75)  node { $x^2_1$};
\filldraw [black] (2.5,0) circle (1.2pt)
                  (2.5,1.5) circle (1.2pt)
                  (2.5,3) circle (1.2pt);
\filldraw [black] (4,0) circle (1.2pt)
                  (4,1.5) circle (1.2pt)
                  (4,3) circle (1.2pt);
\draw  (2.5,0)--(4,0)--(4,3)--(2.5,3)--(2.5,0);
\draw  (2.75,0.25)  node { $x^1_2$};\draw  (3.8,0.25)  node { $z^2_2$};
\draw  (2.75,1.5)  node { $y^2_2$};\draw  (3.8,1.5)  node { $y^1_2$};
\draw  (2.75,2.75)  node { $z^1_2$};\draw  (3.8,2.75)  node { $x^2_2$};
\filldraw [black] (5,0) circle (1.2pt)
                  (5,1.5) circle (1.2pt)
                  (5,3) circle (1.2pt);
\filldraw [black] (6.5,0) circle (1.2pt)
                  (6.5,1.5) circle (1.2pt)
                  (6.5,3) circle (1.2pt);
\draw  (5,0)--(6.5,0)--(6.5,3)--(5,3)--(5,0);
\draw  (5.25,0.25)  node { $x^1_3$};\draw  (6.3,0.25)  node { $z^2_3$};
\draw  (5.25,1.5)  node { $y^2_3$};\draw  (6.3,1.5)  node { $y^1_3$};
\draw  (5.25,2.75)  node { $z^1_3$};\draw  (6.3,2.75)  node { $x^2_3$};
\filldraw [black] (7.5,0) circle (1.2pt)
                  (7.5,1.5) circle (1.2pt)
                  (7.5,3) circle (1.2pt);
\filldraw [black] (9,0) circle (1.2pt)
                  (9,1.5) circle (1.2pt)
                  (9,3) circle (1.2pt);
\draw  (7.5,0)--(9,0)--(9,3)--(7.5,3)--(7.5,0);
\draw  (7.75,0.25)  node { $x^1_4$};\draw  (8.8,0.25)  node { $z^2_4$};
\draw  (7.75,1.5)  node { $y^2_4$};\draw  (8.8,1.5)  node { $y^1_4$};
\draw  (7.75,2.75)  node { $z^1_4$};\draw  (8.8,2.75)  node { $x^2_4$};
\filldraw [black] (10,0) circle (1.2pt)
                  (10,1.5) circle (1.2pt)
                  (10,3) circle (1.2pt);
\filldraw [black] (11.5,0) circle (1.2pt)
                  (11.5,1.5) circle (1.2pt)
                  (11.5,3) circle (1.2pt);
\draw  (10,0)--(11.5,0)--(11.5,3)--(10,3)--(10,0);
\draw  (10.25,0.25)  node { $x^1_5$};\draw  (11.3,0.25)  node { $z^2_5$};
\draw  (10.25,1.5)  node { $y^2_5$};\draw  (11.3,1.5)  node { $y^1_5$};
\draw  (10.25,2.75)  node { $z^1_5$};\draw  (11.3,2.75)  node { $x^2_5$};
\filldraw [black] (12.5,0) circle (1.2pt)
                  (12.5,1.5) circle (1.2pt)
                  (12.5,3) circle (1.2pt);
\filldraw [black] (14,0) circle (1.2pt)
                  (14,1.5) circle (1.2pt)
                  (14,3) circle (1.2pt);
\draw  (12.5,0)--(14,0)--(14,3)--(12.5,3)--(12.5,0);
\draw  (12.75,0.25)  node { $x^1_6$};\draw  (13.8,0.25)  node { $z^2_6$};
\draw  (12.75,1.5)  node { $y^2_6$};\draw  (13.8,1.5)  node { $y^1_6$};
\draw  (12.75,2.75)  node { $z^1_6$};\draw  (13.8,2.75)  node { $x^2_6$};
\filldraw [black] (15,0) circle (1.2pt)
                  (15,1.5) circle (1.2pt)
                  (15,3) circle (1.2pt);
\filldraw [black] (16.5,0) circle (1.2pt)
                  (16.5,1.5) circle (1.2pt)
                  (16.5,3) circle (1.2pt);
\draw  (15,0)--(16.5,0)--(16.5,3)--(15,3)--(15,0);
\draw  (15.25,0.25)  node { $x^1_7$};\draw  (16.3,0.25)  node { $z^2_7$};
\draw  (15.25,1.5)  node { $y^2_7$};\draw  (16.3,1.5)  node { $y^1_7$};
\draw  (15.25,2.75)  node { $z^1_7$};\draw  (16.3,2.75)  node { $x^2_7$};
\filldraw [black] (17.5,0) circle (1.2pt)
                  (17.5,1.5) circle (1.2pt)
                  (17.5,3) circle (1.2pt);
\filldraw [black] (19,0) circle (1.2pt)
                  (19,1.5) circle (1.2pt)
                  (19,3) circle (1.2pt);
\draw  (17.5,0)--(19,0)--(19,3)--(17.5,3)--(17.5,0);
\draw  (17.75,0.25)  node { $x^1_8$};\draw  (18.8,0.25)  node { $z^2_8$};
\draw  (17.75,1.5)  node { $y^2_8$};\draw  (18.8,1.5)  node { $y^1_8$};
\draw  (17.75,2.75)  node { $z^1_8$};\draw  (18.8,2.75)  node { $x^2_8$};
\end{tikzpicture}
\\ (e) The graph $G_3$
\caption{A planar decomposition of $K_{8,8,8}\times K_2$}
\label{figure 8}
\end{center}
\end{figure}

\begin{lemma}\cite{BM08}\label{le4} Let $G$ be a planar graph, and let $f$ be a face in some planar embedding of $G$. Then $G$ admits a planar embedding whose outer face has the same boundary as $f$.
\end{lemma}

\begin{lemma}\label{le5} When $n=4p+1$, there exists a planar decomposition of the Kronecker product graph $K_{n,n,n}\times K_2$ with $2p+1$ subgraphs.
\end{lemma}
\begin{proof}

{\bf Case 1.}~  When $p\leq 1$.

When $p=0$, the Kronecker product graph $K_{1,1,1}\times K_2$ is a cycle of length $6$, so $K_{1,1,1}\times K_2$ is a planar graph. When $p=1$, as shown in Figure \ref{figure 9}, we give a planar decomposition of $K_{5,5,5}\times K_2$ with three subgraphs $A,B$  and $C$.
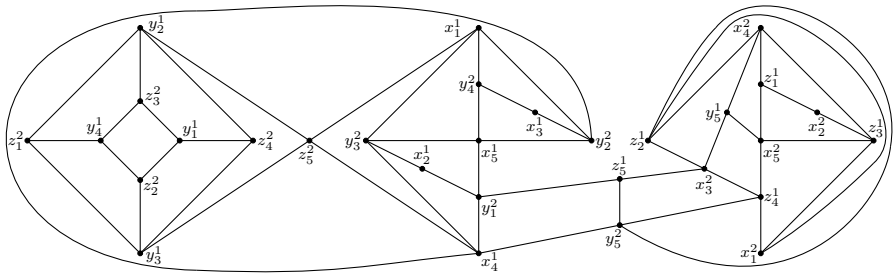
\begin{figure}[H]
\begin{center}
\begin{tikzpicture}
[scale=0.75]
\tikzstyle{every node}=[font=\tiny,scale=0.75]
\tiny
[inner sep=0pt]
\filldraw [black] (0,0) circle (1.2pt)
                  (1.3,0) circle (1.2pt)
                  (2.7,0) circle (1.2pt)
                  (4,0) circle (1.2pt);
\filldraw [black] (2,0.7) circle (1.2pt)
                  (2,2) circle (1.2pt)
                  (2,-0.7) circle (1.2pt)
                  (2,-2) circle (1.2pt);
\draw (0,0)--(2,-2)--(4,0)--(2,2)--(0,0);
\draw (2,0.7)--(1.3,0)--(2,-0.7)--(2.7,0)--(2,0.7);
\draw (0,0)--(1.3,0);\draw (2.7,0)--(4,0);
\draw (2,0.7)--(2,2);\draw (2,-0.7)--(2,-2);
\draw (-0.2,0)  node { $z^2_1$};\draw (2.2,-0.8)  node { $z^2_2$};
\draw (2.23,0.8)  node { $z^2_3$};\draw (4.23,0)  node { $z^2_4$};
\draw (2.9,0.2)  node { $y^1_1$};\draw (2.3,2.07)  node { $y^1_2$};
\draw (2.25,-2.03)  node { $y^1_3$};\draw (1.2,0.25)  node { $y^1_4$};
\filldraw [black] (6,0) circle (1.2pt)
                  (8,-2) circle (1.2pt)
                  (8,2) circle (1.2pt)
                  (10,0) circle (1.2pt);
\filldraw [black] (8,0) circle (1.2pt)
                  (8,1) circle (1.2pt)
                  (8,-1) circle (1.2pt);
\filldraw [black] (9,0.5) circle (1.2pt)
                  (7,-0.5) circle (1.2pt);
\draw (6,0)--(8,-2);\draw(10,0)--(8,2)--(6,0);
\draw (8,2)--(8,-2);\draw (6,0)--(10,0);
\draw (6,0)--(8,-1);\draw (10,0)--(8,1);
\draw[-] (10,0)..controls+(-0.1,3.1)and+(1.5,0)..(3,2.3);
\draw[-] (3,2.3)..controls+(-4.5,-0.01)and+(-4.5,0.1)..(3,-2.3);
\draw[-] (3,-2.3)..controls+(2.5,-0.1)and+(-2,-0.2)..(8,-2);
\draw (8.2,-1.2)  node { $y^2_1$};\draw (10.22,0)  node { $y^2_2$};
\draw (5.78,0)  node { $y^2_3$};\draw (7.8,1)  node { $y^2_4$};
\draw (7.55,2)  node { $x^1_1$};\draw (7,-0.25)  node { $x^1_2$};
\draw (9,0.25)  node { $x^1_3$};\draw (8.2,-2.15)  node { $x^1_4$};
\filldraw [black] (11,0) circle (1.2pt)
                  (13,-2) circle (1.2pt)
                  (13,2) circle (1.2pt)
                  (15,0) circle (1.2pt);
\filldraw [black] (13,0) circle (1.2pt)
                  (13,1) circle (1.2pt)
                  (13,-1) circle (1.2pt);
\filldraw [black] (14,0.5) circle (1.2pt)
                  (12,-0.5) circle (1.2pt);
\draw(13,-2)--(15,0)--(13,2)--(11,0);
\draw (13,2)--(13,-2);\draw (13,0)--(15,0);
\draw (11,0)--(13,-1);\draw (15,0)--(13,1);
\draw[-] (11,0)..controls+(0.6,1)and+(-0.1,-0.1)..(12.4,2);
\draw[-] (13,-2)..controls+(1,0.6)and+(-0.1,-0.1)..(15,-0.4);
\draw[-] (12.4,2)..controls+(0.9,1)and+(1,0.9)..(15,-0.4);
\draw (13.2,1.12)  node { $z^1_1$};\draw (10.82,0)  node { $z^1_2$};
\draw (15.05,0.22)  node { $z^1_3$};\draw (13.2,-1)  node { $z^1_4$};
\draw (12.8,-2)  node { $x^2_1$};\draw (14,0.25)  node { $x^2_2$};
\draw (12,-0.78)  node { $x^2_3$};\draw (12.67,2)  node { $x^2_4$};
\filldraw [black] (5,0) circle (1.2pt);
\draw (4.95,-0.25)  node { $z^2_5$};
\draw (5,0)--(2,2);\draw (5,0)--(2,-2);
\draw (5,0)--(8,2);\draw (5,0)--(8,-2);
\filldraw [black] (12.4,0.5) circle (1.2pt);
\draw (12.17,0.5)  node { $y^1_5$};
\draw (13,2)--(12,-0.5);\draw (12.4,0.5)--(13,0);
\filldraw [black] (10.5,-0.68) circle (1.2pt);
\draw (10.5,-0.44)  node { $z^1_5$};
\draw (8,-1)--(12,-0.5);
\filldraw [black] (10.5,-1.5) circle (1.2pt);
\draw (10.4,-1.74)  node { $y^2_5$};
\draw (8,-2)--(13,-1);\draw (10.5,-0.68)--(10.5,-1.5);
\draw (8.2,-0.22)  node { $x^1_5$};
\draw (13.2,-0.22)  node { $x^2_5$};
\draw[-] (11,0)..controls+(0.6,1.3)and+(-0.2,-0.2)..(12.2,2);
\draw[-] (10.5,-1.5)..controls+(3,-1.9)and+(-0.1,-1)..(15.3,-0.1);
\draw[-] (12.2,2)..controls+(0.9,1.1)and+(0.3,1.8)..(15.3,-0.1);
\end{tikzpicture}
\\ (a) The graph $A$
\end{center}
\end{figure}
\begin{figure}[H]
\begin{center}
\begin{tikzpicture}
[scale=0.75]
\tikzstyle{every node}=[font=\tiny,scale=0.75]
\tiny
[inner sep=0pt]
\filldraw [black] (0,0) circle (1.2pt)
                  (1.3,0) circle (1.2pt)
                  (2.7,0) circle (1.2pt)
                  (4,0) circle (1.2pt);
\filldraw [black] (2,0.7) circle (1.2pt)
                  (2,2) circle (1.2pt)
                  (2,-0.7) circle (1.2pt)
                  (2,-2) circle (1.2pt);
\draw (0,0)--(2,-2)--(4,0)--(2,2)--(0,0);
\draw (2,0.7)--(1.3,0)--(2,-0.7)--(2.7,0)--(2,0.7);
\draw (0,0)--(1.3,0);\draw (2.7,0)--(4,0);
\draw (2,0.7)--(2,2);\draw (2,-0.7)--(2,-2);
\draw (-0.2,0)  node { $z^1_1$};\draw (2.2,-0.8)  node { $z^1_2$};
\draw (2.23,0.8)  node { $z^1_3$};\draw (4.23,0)  node { $z^1_4$};
\draw (2.9,0.2)  node { $y^2_1$};\draw (2.3,2.07)  node { $y^2_2$};
\draw (2.25,-2.03)  node { $y^2_3$};\draw (1.2,0.25)  node { $y^2_4$};
\filldraw [black] (6,0) circle (1.2pt)
                  (8,-2) circle (1.2pt)
                  (8,2) circle (1.2pt)
                  (10,0) circle (1.2pt);
\filldraw [black] (8,0) circle (1.2pt)
                  (8,1) circle (1.2pt)
                  (8,-1) circle (1.2pt);
\filldraw [black] (9,0.5) circle (1.2pt)
                  (7,-0.5) circle (1.2pt);
\draw (6,0)--(8,-2);\draw(10,0)--(8,2)--(6,0);
\draw (8,2)--(8,-2);\draw (6,0)--(10,0);
\draw (6,0)--(8,-1);\draw (10,0)--(8,1);
\draw[-] (10,0)..controls+(-0.1,3.1)and+(1.5,0)..(3,2.3);
\draw[-] (3,2.3)..controls+(-4.5,-0.01)and+(-4.5,0.1)..(3,-2.3);
\draw[-] (3,-2.3)..controls+(2.5,-0.1)and+(-2,-0.2)..(8,-2);
\draw (8.2,-1.2)  node { $y^1_4$};\draw (10.22,0)  node { $y^1_3$};
\draw (5.78,0)  node { $y^1_2$};\draw (7.8,1)  node { $y^1_1$};
\draw (7.55,2)  node { $x^2_4$};\draw (7,-0.25)  node { $x^2_3$};
\draw (9,0.25)  node { $x^2_2$};\draw (8.2,-2.15)  node { $x^2_1$};
\filldraw [black] (11,0) circle (1.2pt)
                  (13,-2) circle (1.2pt)
                  (13,2) circle (1.2pt)
                  (15,0) circle (1.2pt);
\filldraw [black] (13,0) circle (1.2pt)
                  (13,1) circle (1.2pt)
                  (13,-1) circle (1.2pt);
\filldraw [black] (14,0.5) circle (1.2pt)
                  (12,-0.5) circle (1.2pt);
\draw(13,-2)--(15,0)--(13,2)--(11,0);
\draw (13,2)--(13,-2);\draw (13,0)--(15,0);
\draw (11,0)--(13,-1);\draw (15,0)--(13,1);
\draw[-] (11,0)..controls+(0.6,1)and+(-0.1,-0.1)..(12.4,2);
\draw[-] (13,-2)..controls+(1,0.6)and+(-0.1,-0.1)..(15,-0.4);
\draw[-] (12.4,2)..controls+(0.9,1)and+(1,0.9)..(15,-0.4);
\draw (13.2,1.12)  node { $z^2_4$};\draw (10.82,0)  node { $z^2_3$};
\draw (15.05,0.22)  node { $z^2_2$};\draw (13.2,-1)  node { $z^2_1$};
\draw (12.8,-2)  node { $x^1_4$};\draw (14,0.25)  node { $x^1_3$};
\draw (12,-0.78)  node { $x^1_2$};\draw (12.67,2)  node { $x^1_1$};
\filldraw [black] (5,0) circle (1.2pt);
\draw (4.95,-0.25)  node { $z^1_5$};
\draw (5,0)--(2,2);\draw (5,0)--(2,-2);
\draw (5,0)--(8,2);\draw (5,0)--(8,-2);
\filldraw [black] (12.4,0.5) circle (1.2pt);
\draw (12.17,0.5)  node { $y^2_5$};
\draw (13,2)--(12,-0.5);\draw (12.4,0.5)--(13,0);
\filldraw [black] (10.5,-0.68) circle (1.2pt);
\draw (10.5,-0.44)  node { $z^2_5$};
\draw (8,-1)--(12,-0.5);
\filldraw [black] (10.5,-1.5) circle (1.2pt);
\draw (10.4,-1.74)  node { $y^1_5$};
\draw (8,-2)--(13,-1);\draw (10.5,-0.68)--(10.5,-1.5);
\draw (8.2,-0.22)  node { $x^2_5$};
\draw (13.2,-0.22)  node { $x^1_5$};
\draw[-] (11,0)..controls+(0.6,1.3)and+(-0.2,-0.2)..(12.2,2);
\draw[-] (10.5,-1.5)..controls+(3,-1.9)and+(-0.1,-1)..(15.3,-0.1);
\draw[-] (12.2,2)..controls+(0.9,1.1)and+(0.3,1.8)..(15.3,-0.1);
\end{tikzpicture}
\\ (b) The graph $B$
\end{center}
\end{figure}
\begin{figure}[H]
\begin{center}
\begin{tikzpicture}
[scale=0.75]
\tikzstyle{every node}=[font=\tiny,scale=0.75]
\tiny
[inner sep=0pt]
\filldraw [black] (0,0) circle (1.2pt)
                  (1,0) circle (1.2pt)
                  (2,0) circle (1.2pt)
                  (3,0) circle (1.2pt)
                  (4,0) circle (1.2pt)
                  (5,0) circle (1.2pt)
                  (6,0) circle (1.2pt);
\filldraw [black] (0,2) circle (1.2pt)
                  (1,2) circle (1.2pt)
                  (2,2) circle (1.2pt)
                  (3,2) circle (1.2pt)
                  (4,2) circle (1.2pt)
                  (5,2) circle (1.2pt)
                  (6,2) circle (1.2pt);
\filldraw [black] (2,1.33) circle (1.2pt);
\draw (0,0)--(6,0); \draw (0,2)--(6,2); \draw (0,0)--(0,2);\draw (6,0)--(6,2);                                                  \draw (0,0)--(2,2);\draw (3,2)--(3,0)--(5,2);
\draw (0,0)--(3,2);\draw (2,0)--(2,1.33);
\draw[-] (0,0)..controls+(1,-0.3)and+(-1,-0.3)..(5,0);
\draw[-] (0,2)..controls+(1,0.3)and+(-1,0.3)..(5,2);
\draw[-] (0,0)..controls+(0.2,-0.3)and+(-1,-0.01)..(5,-0.4);
\draw[-] (5,-0.4)..controls+(1.5,0.01)and+(0.5,-1.7)..(6,2);
\draw (-0.2,0.25)  node { $x^1_3$};
\draw (1,0.25)  node { $y^2_3$};
\draw (1.8,0.25)  node { $z^1_3$};
\draw (2.8,0.25)  node { $x^2_3$};
\draw (4,0.25)  node { $y^1_3$};
\draw (5,0.25)  node { $z^2_3$};
\draw (5.8,0.25)  node { $x^1_5$};
\draw (-0.2,1.75)  node { $z^2_1$};
\draw (1,1.75)  node { $x^1_1$};
\draw (2.05,1.75)  node { $y^2_1$};
\draw (3.2,1.75)  node { $z^1_1$};
\draw (4,1.75)  node { $x^2_1$};
\draw (5.05,1.75)  node { $y^1_1$};
\draw (5.8,1.75)  node { $z^2_5$};
\draw (2.2,1.25)  node { $y^2_5$};
\filldraw [black] (7,0) circle (1.2pt)
                  (8,0) circle (1.2pt)
                  (9,0) circle (1.2pt)
                  (10,0) circle (1.2pt)
                  (11,0) circle (1.2pt)
                  (12,0) circle (1.2pt)
                  (13,0) circle (1.2pt);
\filldraw [black] (7,2) circle (1.2pt)
                  (8,2) circle (1.2pt)
                  (9,2) circle (1.2pt)
                  (10,2) circle (1.2pt)
                  (11,2) circle (1.2pt)
                  (12,2) circle (1.2pt)
                  (13,2) circle (1.2pt);
\filldraw [black] (9,1.33) circle (1.2pt);
\draw (7,0)--(13,0); \draw (7,2)--(13,2); \draw (7,0)--(7,2);\draw (13,0)--(13,2);                                                  \draw (7,0)--(9,2);\draw (10,2)--(10,0)--(12,2);
\draw (7,0)--(10,2);\draw (9,0)--(9,1.33);
\draw[-] (7,0)..controls+(1,-0.3)and+(-1,-0.3)..(12,0);
\draw[-] (7,2)..controls+(1,0.3)and+(-1,0.3)..(12,2);
\draw[-] (7,0)..controls+(0.2,-0.3)and+(-1,-0.01)..(12,-0.4);
\draw[-] (12,-0.4)..controls+(1.5,0.01)and+(0.5,-1.7)..(13,2);
\draw (6.8,0.25)  node { $x^2_2$};
\draw (8,0.25)  node { $y^1_2$};
\draw (8.8,0.25)  node { $z^2_2$};
\draw (9.8,0.25)  node { $x^1_2$};
\draw (11,0.25)  node { $y^2_2$};
\draw (12,0.25)  node { $z^1_2$};
\draw (12.8,0.25)  node { $x^2_5$};
\draw (6.8,1.75)  node { $z^1_4$};
\draw (8,1.75)  node { $x^2_4$};
\draw (9.05,1.75)  node { $y^1_4$};
\draw (10.2,1.75)  node { $z^2_4$};
\draw (11,1.75)  node { $x^1_4$};
\draw (12.05,1.75)  node { $y^2_4$};
\draw (12.8,1.75)  node { $z^1_5$};
\draw (9.2,1.25)  node { $y^1_5$};
\end{tikzpicture}
\\ (c) The graph $C$
\caption{A planar decomposition of $K_{5,5,5}\times K_2$}
\label{figure 9}
\end{center}
\end{figure}

{\bf Case 2.}~  When $p\geq 2$.

Suppose that $\{G^1_1,\dots,G^1_p,G^2_1,\dots,G^2_p,G_{p+1}\}$ is the planar decomposition of $K_{4p,4p,4p}\times K_2$ as provided in the proof of Lemma \ref{le3}. By adding vertices $x^1_{4p+1}$, $x^2_{4p+1}$, $y^1_{4p+1}$, $y^2_{4p+1}$, $z^1_{4p+1}$, $z^2_{4p+1}$ to each graph in this decomposition, and some modifications of adding and deleting edges to these graphs, a planar decomposition of $K_{4p+1,4p+1,4p+1}\times K_2$ will be obtained.

 For convenience, in Figure \ref{figure 2} we label some faces of $G_r (1\leq r\leq p)$ with face $1, 2$ and $3$.  As indicated in Figure \ref{figure 2}, the face $1$ is bounded by $v_{4r-1}u_{4r-3}v_{4r-2}u_{4r}$, the face $3$ is its outer face, bounded by $v_{4r-3}u_{4r-2}v_{4r}u_{4r-1}$. The face 2 is bounded by $u_{4r-3}v_{4r-1}u_{4r-2}v_{j}$ in which vertex $v_j$ can be any vertex of $\bigcup\limits^p_{i=1,i\neq r}\{v_{4i-2},v_{4i}\}$. Because $u_{4r-3}$ and $u_{4r-2}$ in $G_{r}$ $(1\leq r\leq p)$ is joined by $2p-2$ edge-disjoint paths of length two that we call parallel paths, we can change the order of these parallel paths without changing the planarity of $G_{r}$. Analogously, we can change the order of parallel paths between $u_{4r-1}$ and $u_{4r}$, $v_{4r-3}$ and $v_{4r-1}$, $v_{4r-2}$ and $v_{4r}$. In addition, the subscripts of all the vertices are taken module $4p$, except that of the new added vertices $x^1_{4p+1}$, $x^2_{4p+1}$, $y^1_{4p+1}$, $y^2_{4p+1}$, $z^1_{4p+1}$ and  $z^2_{4p+1}$.

\noindent {\bf Step 1:}~~Add the vertices $x^1_{4p+1}$ and $y^2_{4p+1}$ to graph $G_r(X^1,Y^2)$.

Place vertices $x^1_{4p+1}$ and $y^2_{4p+1}$ in face $1$ and face $2$ of $G_r(X^1,Y^2)$, respectively. Join $x^1_{4p+1}$ to vertices $y^2_{4r-3}$ and $y^2_{4r}$. Change the order of the parallel paths between $y_{4r-2}^2$ and $y_{4r-3}^2$, such that $x^1_{4r+2}\in \bigcup\limits^p_{i=1,i\neq r}\{x^1_{4i-2},x^1_{4i}\}$ are incident with the face $2$, and join $y^2_{4p+1}$ to both $x^1_{4r-1}$ and $x^1_{4r+2}$.

\noindent {\bf Step 2:}~~Add the vertices $x^2_{4p+1}$ and $y^1_{4p+1}$ to graph $G_r(X^2,Y^1)$.

Similar to step $1$, place $x^2_{4p+1}$ and $y^1_{4p+1}$ in face $1$ and face $2$ of $G_r(X^2,Y^1)$, respectively. Join $x^2_{4p+1}$ to both $y^1_{4r-3}$ and $y^1_{4r}$, join $y^1_{4p+1}$ to both $x^2_{4r-1}$ and $x^2_{4r+2}\in \bigcup\limits^p_{i=1,i\neq r}\{x^2_{4i-2},x^2_{4i}\}$.

\noindent {\bf Step 3:}~~Add the vertices $y^1_{4p+1}$ and $z^2_{4p+1}$ to graph $G_r(Y^1,Z^2)$.

Place $y^1_{4p+1}$ in face $3$ of $G_r(Y^1,Z^2)$ and join it to vertices $z^2_{4r-2}$ and $z^2_{4r-1}$.
Place $z^2_{4p+1}$ in face $1$ of $G_r(Y^1,Z^2)$ and join it to vertices $y^1_{4r-2}$ and $y^1_{4r-1}$.

\noindent {\bf Step 4:}~~Add the vertices $y^2_{4p+1}$ and $z^1_{4p+1}$ to graph $G_r(Y^2,Z^1)$.

Place $y^2_{4p+1}$ in face $3$ of $G_r(Y^2,Z^1)$ and join it to vertices $z^1_{4r-2}$ and $z^1_{4r-1}$.
Place $z^1_{4p+1}$ in face $1$ of $G_r(Y^2,Z^1)$ and join it to vertices $y^2_{4r-2}$ and $y^2_{4r-1}$.

\noindent {\bf Step 5:}~~ Add the vertices $z^1_{4p+1}$ and $x^2_{4p+1}$ to graph $G_r(Z^1,X^2)$.

Place $z^1_{4p+1}$ in face $1$ of $G_r(Z^1,X^2)$ and join it to vertices $x^2_{4r-3}$ and $x^2_{4r}$.
Place $x^2_{4p+1}$ in face $3$ of $G_r(Z^1,X^2)$ and join it to vertices $z^1_{4r-3}$ and $z^1_{4r}$.

\noindent {\bf Step 6:}~~ Add the vertices $z^2_{4p+1}$ and $x^1_{4p+1}$ to graph $G_r(Z^2,X^1)$.

Place $z^2_{4p+1}$ in face $1$ of $G_r(Z^2,X^1)$ and join it to vertices $x^1_{4r-3}$ and $x^1_{4r}$.
Place $x^1_{4p+1}$ in face $3$ of $G_r(Z^2,X^1)$ and join it to vertices $z^2_{4r-3}$ and $z^2_{4r}$.

We denote the above graphs we obtain from Steps $1-6$ by $\widehat{G}_r(X^1,Y^2)$, $\widehat{G}_r(X^2,Y^1)$, $\widehat{G}_r(Y^1,Z^2)$, $\widehat{G}_r(Y^2,Z^1)$, $\widehat{G}_r(Z^1,X^2)$ and $\widehat{G}_r(Z^2,X^1)$ respectively.

Let $$\widehat{G}^1_r=\widehat{G}_r(X^1,Y^2)\cup \widehat{G}_r(Y^1,Z^2)\cup \widehat{G}_r(Z^1,X^2)$$
and $$\widehat{G}^2_{r}=\widehat{G}_{r}(X^2,Y^1)\cup \widehat{G}_r(Y^2,Z^1)\cup \widehat{G}_r(Z^2,X^1).$$

\noindent {\bf Step 7:}~~ Add the edges
$z^1_{4r}x^2_{4r}$, $y^1_{4r-1}z^2_{4r-1}$, $z^1_{4r-2}y^2_{4r-2}$, $x^1_{4r-3}z^2_{4r-3}$ and
$z^2_{4r}x^1_{4r}$, $y^2_{4r-1}z^1_{4r-1}$, $z^2_{4r-2}y^1_{4r-2}$, $x^2_{4r-3}z^1_{4r-3}$ to graphs $\widehat{G}^1_r$ and $\widehat{G}^2_r$ respectively, $1\leq r\leq p$.

For graph $\widehat{G}_r(Y^1,Z^2)\subset \widehat{G}^1_r$,  we delete the edge $y^1_{4r-3}z^2_{4r}$ and join the vertex $y^1_{4r-1}$ to vertex $z^2_{4r-1}$, then we get a planar graph $\widetilde{G}_r(Y^1,Z^2)$. According to Lemma \ref{le4}, the graph $\widetilde{G}_r(Y^1,Z^2)$ has a planar embedding whose outer face has the same boundary as face $2$, then the vertex $z^2_{4r-3}$ is on the boundary of this outer face.

For graph $\widehat{G}_r(Z^1,X^2)\subset \widehat{G}^1_r$, delete the edge $z^1_{4r-2}x^2_{4r-1}$ and join $z^1_{4r}$ to $x^2_{4r}$, then we get a planar graph $\widetilde{G}_r(Z^1,X^2)$. According to Lemma \ref{le4}, the graph $\widetilde{G}_r(Z^1,X^2)$ has a planar embedding whose outer face has boundary as \\$z^1_{4r}x^2_{4r}z^1_{4r-2}x^2_{i}z^1_{4r}~(x^2_{i}\in \bigcup\limits^p_{i=1,i\neq r}\{x_{4i-1}^2,x_{4i}^2\})$, then the vertex $z^1_{4r-2}$ is on the boundary of this outer face.

Since the vertices $x^1_{4r-3}$ and  $y^2_{4r-2}$ are on the boundary of the outer face of the embedding of  $\widehat{G}_r(X^1,Y^2)\subset \widehat{G}^1_r$, we can join $x^1_{4r-3}$ to $z^2_{4r-3}$, $y^2_{4r-2}$ to $z^1_{4r-2}$ without edge crossing. Then we get a planar graph $\widetilde{G}^1_r$.

With the same process, for the graph ${G}^2_r$, we delete edges $y^2_{4r-3}z^1_{4r}$ and $z^2_{4r-2}x^1_{4r-1}$, join $y^2_{4r-1}$ to $z^1_{4r-1}$, join $z^2_{4r}$ to  $x^1_{4r}$, join $x^2_{4r-3}$ to $z^1_{4r-3}$ and join $y^1_{4r-2}$ to $z^2_{4r-2}$, then we get a planar graph $\widetilde{G}^2_r$.

Table \ref{tab 1} shows the edges that we add to $G^1_{r}$ and $G^2_{r}$ $(1\leq r\leq p)$ in Steps $1-7$.

\begin{table}[H]
\centering\caption {The edges we add to $G^1_{r}$ and $G^2_{r}$ $(1\leq r\leq p)$ }\label{tab 1}\tiny
\renewcommand{\arraystretch}{2.5}
\begin{tabular}{|*{9}{c|}}
\hline
\multicolumn{8}{|c|}{edges} &  \multicolumn{1}{c|}{subscript} \\
\hline
\multicolumn{2}{|c|}{$x^1_{4p+1}y^2_i,x^2_{4p+1}y^1_i$}   &
\multicolumn{2}{c|}{$z^1_{4p+1}x^2_i,z^2_{4p+1}x^1_i$ }   &
\multicolumn{2}{c|}{$x^1_{4p+1}z^2_i,x^2_{4p+1}z^1_i$}    &
\multicolumn{2}{c|}{$x^1_{i}z^2_i,x^2_{i}z^1_i$ }  &{$i=4r-3,4r$.}      \\
\hline
\multicolumn{2}{|c|}{$y^1_{4p+1}z^2_i,y^2_{4p+1}z^1_i$}   & \multicolumn{2}{c|}{$z^1_{4p+1}y^2_i,z^2_{4p+1}y^1_i$ }&
\multicolumn{2}{c|}{$y^1_{4p+1}x^2_i,y^2_{4p+1}x^1_i$}   & \multicolumn{2}{c|}{ $y^1_{i}z^2_i,y^2_{i}z^1_i$ }   & {$i=4r-2,4r-1$.}      \\
\hline
\end{tabular}
\end{table}

\noindent {\bf Step 8:}~~ The remaining edges form a planar graph $\widetilde{G}_{p+1}.$

The edges that belong to $K_{4p+1,4p+1,4p+1}\times K_2$ but not to any
$\widetilde{G}^1_r, \widetilde{G}^2_r$ $(1\leq r\leq p)$ are shown in Table \ref{tab 2}, in which the edges in the last two rows list the edges deleted in Step $7$. The remaining edges form a graph, denote by $\widetilde{G}_{p+1}$. We draw a planar embedding of $\widetilde{G}_{p+1}$ in Figure \ref{figure 4}, so $\widetilde{G}_{p+1}$ is a planar graph.

Therefore $\{\widetilde{G}^1_1,\dots,\widetilde{G}^1_{p},\widetilde{G}^2_1,\dots,\widetilde{G}^2_{p},\widetilde{G}_{p+1}\}$ is a planar decomposition of \\$K_{4p+1,4p+1,4p+1}\times K_2$, the Lemma follows.
\end{proof}

Figure \ref{figure 5} illustrates a planar decomposition of $K_{9,9,9}\times K_2$ with five subgraphs.

\begin{table}[H]
\centering\caption {The edges of $\widetilde{G}_{p+1}$} \label{tab 2}
\tiny
\renewcommand{\arraystretch}{2.5}

\caption{The graph $\widetilde{G}_{p+1}$}
\label{figure 4}
\end{center}
\end{figure}

\begin{figure}[H]
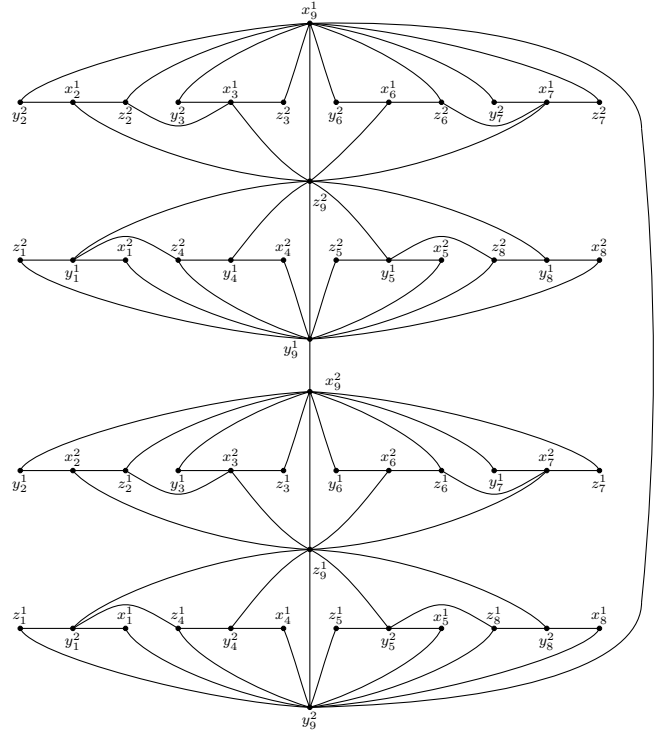

\begin{center}
%
\\ (e) The graph $\widetilde{G}_3$
\caption{A planar decomposition of $K_{9,9,9}\times K_2$}
\label{figure 5}
\end{center}
\end{figure}

A graph $G$ is said to be thickness {\it $t$-minimal}, if $\theta(G)=t$ and every proper subgraphs of it have a thickness less than $t$.

\begin{lemma}\label{le6} When $n=4p+3$, there exists a planar decomposition of Kronecker product graph $K_{4p+3,4p+3,4p+3}\times K_2$ with $2p+2$ subgraphs.
\end{lemma}
\begin{proof}

{\bf Case 1.} When $p=0$.

As shown in Figure \ref{figure 10}, we give a planar decomposition of $K_{3,3,3}\times K_2$  with $2$ subgraphs.

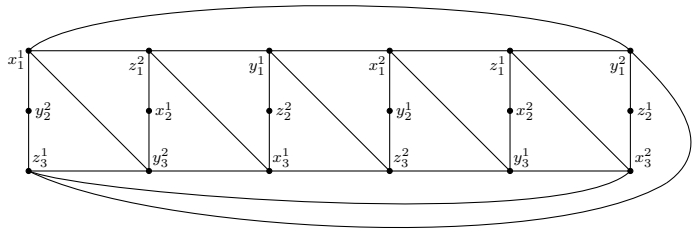
\begin{figure}[H]
\begin{center}
\begin{tikzpicture}
[scale=0.8]
\tikzstyle{every node}=[font=\tiny,scale=0.75]
\tiny
[inner sep=0pt]
\filldraw [black] (0,0) circle (1.2pt)
                  (2,0) circle (1.2pt)
                  (4,0) circle (1.2pt)
                  (6,0) circle (1.2pt)
                  (8,0) circle (1.2pt)
                  (10,0) circle (1.2pt);
\filldraw [black] (0,1) circle (1.2pt)
                  (2,1) circle (1.2pt)
                  (4,1) circle (1.2pt)
                  (6,1) circle (1.2pt)
                  (8,1) circle (1.2pt)
                  (10,1) circle (1.2pt);
\filldraw [black] (0,2) circle (1.2pt)
                  (2,2) circle (1.2pt)
                  (4,2) circle (1.2pt)
                  (6,2) circle (1.2pt)
                  (8,2) circle (1.2pt)
                  (10,2) circle (1.2pt);
\draw (0,2)--(2,0);\draw (2,2)--(4,0);\draw (4,2)--(6,0);\draw (6,2)--(8,0);\draw (8,2)--(10,0);
\draw (0,2)--(0,0);\draw (2,2)--(2,0);\draw (4,2)--(4,0);\draw (6,2)--(6,0);\draw (8,2)--(8,0);
\draw (10,2)--(10,0);
\draw (0,0)--(10,0);\draw (0,2)--(10,2);
\draw[-] (0,2)..controls+(1,1)and+(-1,1)..(10,2);
\draw[-] (0,0)..controls+(1,-0.4)and+(-1,-1)..(10,0);
\draw[-] (0,0)..controls+(2,-1)and+(-2,-1.2)..(10.6,-0.2);
\draw[-] (10.6,-0.2)..controls+(1,0.7)and+(0.5,-0.5)..(10,2);
\draw (0.2,0.22)  node { $z^1_{3}$};\draw (2.2,0.22)  node { $y^2_{3}$};
\draw (4.2,0.22)  node { $x^1_{3}$};\draw (6.2,0.22)  node { $z^2_{3}$};
\draw (8.2,0.22)  node { $y^1_{3}$};\draw (10.22,0.22)  node { $x^2_{3}$};
\draw (0.25,1)  node { $y^2_{2}$};\draw (2.25,1)  node { $x^1_{2}$};
\draw (4.25,1)  node { $z^2_{2}$};\draw (6.25,1)  node { $y^1_{2}$};
\draw (8.25,1)  node { $x^2_{2}$};\draw (10.25,1)  node { $z^1_{2}$};
\draw (-0.2,1.85)  node { $x^1_{1}$};\draw (1.8,1.75)  node { $z^2_{1}$};
\draw (3.8,1.75)  node { $y^1_{1}$};\draw (5.8,1.75)  node { $x^2_{1}$};
\draw (7.8,1.75)  node { $z^1_{1}$};\draw (9.8,1.75)  node { $y^2_{1}$};
\end{tikzpicture}
\end{center}
\end{figure}
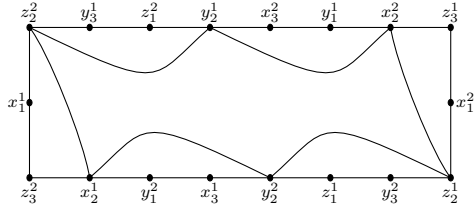
\begin{figure}[H]
\begin{center}
\begin{tikzpicture}
[xscale=0.8]
\tikzstyle{every node}=[font=\tiny,scale=0.8]
[inner sep=0pt]\tiny
\filldraw [black] (0,2) circle (1.2pt)
                     (1,2) circle (1.2pt)
                      (2,2) circle (1.2pt);
\filldraw [black](3,2) circle (1.2pt)
                     (4,2) circle (1.2pt)
                      (5,2) circle (1.2pt);
\filldraw [black](6,2) circle (1.2pt)
                     (7,2) circle (1.2pt)
                      (7,1) circle (1.2pt);
\filldraw [black](0,0) circle (1.2pt)
                     (1,0) circle (1.2pt)
                      (2,0) circle (1.2pt);
\filldraw [black](3,0) circle (1.2pt)
                     (4,0) circle (1.2pt)
                     (5,0) circle (1.2pt);
\filldraw [black](6,0) circle (1.2pt)
                     (7,0) circle (1.2pt)
                     (0,1) circle (1.2pt);
\draw[-](0,2)--(7,2);\draw[-](0,0)--(7,0);
\draw[-](0,2)--(0,0);\draw[-](7,0)--(7,2);
\draw  (0,2)..controls+(0.3,-0.3)and+(-0.1,0.5)..(1,0);
\draw  (1,0)..controls+(1,0.7)and+(-2.3,0.9)..(4,0);
\draw  (4,0)..controls+(1,0.7)and+(-2.3,0.9)..(7,0);
\draw  (7,0)..controls+(-0.3,0.3)and+(0.1,-0.5)..(6,2);
\draw  (0,2)..controls+(2.3,-0.9)and+(-1,-0.7)..(3,2);
\draw  (3,2)..controls+(2.3,-0.9)and+(-1,-0.7)..(6,2);
\draw (0,2.2)node { $z^2_{2}$};\draw (1,2.2)node { $y^1_{3}$};
\draw (3,2.2)node { $y^1_{2}$};\draw (4,2.2)node { $x^2_{3}$};
\draw (6,2.2)node { $x^2_{2}$};\draw (7,2.2)node { $z^1_{3}$};
\draw (7,-0.2)node { $z^1_{2}$};\draw (6,-0.2)node { $y^2_{3}$};
\draw (4,-0.2)node { $y^2_{2}$};\draw (3,-0.2)node { $x^1_{3}$};
\draw (1,-0.2)node { $x^1_{2}$};\draw (0,-0.2)node { $z^2_{3}$};
\draw (-0.2,1)node { $x^1_{1}$};\draw (7.25,1)node { $x^2_{1}$};
\draw (2,2.2)node { $z^2_{1}$};\draw (5,2.2)node { $y^1_{1}$};
\draw (2,-0.2)node { $y^2_{1}$};\draw (5,-0.2)node { $z^1_{1}$};
\end{tikzpicture}
\caption{The planar decomposition of $K_{3,3,3}\times K_2$}
\label{figure 10}
\end{center}
\end{figure}

{\bf Case 2.} When $p\geq 1$.

The graph $K_{4p+3,4p+3}$ is a thickness $(p+2)$-minimal graph. Hobbs, Grossman \cite{HG68} and Bouwer, Broere \cite{BB68}
proved it independently, by giving two different planar subgraphs decompositions $\{H_1,\ldots, H_{p+2}\}$ of $K_{4p+3,4p+3}$ in which $H_{p+2}$ contains only one edge. Suppose that the two vertex parts of $K_{n,n}$ is $\{v_1,\ldots, v_n\}$  and $\{u_1,\ldots, u_n\}$, the only one edge in the $H_{p+2}$ is $v_{a}u_{b}$ (the edge is $v_{1}u_{1}$ in \cite{HG68}  and $v_{4p+3}u_{4p-1}$ in \cite{BB68}). For $1\leq i\leq p+2$, $H_i$ is a bipartite graph, so we also denote it by $H_i(V,U)$.

Because $K_{n,n,n}\times K_2=G^1 \cup G^2$ in which
$G^1=G(X^1,Y^2)\cup G(Y^1,Z^2)\cup G(Z^1,X^2)$ and $G^2=G(X^2,Y^1)\cup G(Y^2,Z^1)\cup G(Z^2,X^1)$, $|X^i|=|Y^i|=|Z^i|=n$ $(i=1,2)$, all the graphs $G(X^1,Y^2), G(Y^1,Z^2), G(Z^1,X^2), G(X^2,Y^1), G(Y^2,Z^1)$ and $G(Z^2,X^1)$ are isomorphic to $K_{n,n}$.

For graph $H_i(V,U)$ $(1\leq i\leq p+2)$, We replace the vertex set $V$ by $X^1$, $U$ by $Y^2$, i.e., for each $1\leq t\leq n$, replace the vertex $v_t$ by $x_t^1$, and $u_t$ by $y_t^2$, then we get a graph $ H_i(X^1,Y^2)$. Analogously, we can obtain graphs $H_i(Y^1,Z^2)$, $H_i(Z^1,X^2)$, $H_i(X^2,Y^1)$, $H_i(Y^2,Z^1)$ and $H_i(Z^2,X^1)$. For $1\leq i\leq p+2$, let
$$H^1_i=H_i(X^1,Y^2)\cup H_i(Y^1,Z^2)\cup H_i(Z^1,X^2),$$ then $H^1_i$ is a planar graph, because $H_i(X^1,Y^2),H_i(Y^1,Z^2) $, $H_i(Z^1,X^2)$ are disjoint with each other. For the same reason, the graph $$H^2_{i}=H_i(X^2,Y^1)\cup H_i(Y^2,Z^1)\cup H_i(Z^2,X^1)$$ is also a planar graph, $1\leq i\leq p+2$. And we have
$$K_{4p+3,4p+3,4p+3}\times K_2 =G^1\cup G^2=\mathop{\cup}\limits^{p+2}_{i=1}(H_i^1\cup H_i^2),$$ in which
$E(H_{p+2}^1)=\{x_a^1y_b^2, y_a^1z_b^2, z_a^1x_b^2\}$ and $E(H_{p+2}^2)=\{x_a^2y_b^1, y_a^2z_b^1, z_a^2x_b^1\}$.

In the following, we will add edges in $E(H_{p+2}^1)$ to graphs $H_1^2$ and $H_2^2$, add edges in $E(H_{p+2}^2)$ to graphs $H_1^1$ and $H_1^2$
to complete the proof. From Lemma \ref{le4}, there exists a planar embedding of $H_1(Y^1,Z^2)$ such that vertex $z_a^2$ on the boundary of its outer face, exists a planar embedding of $H_1(X^1,Y^2)$ such that $x_b^1$ on the boundary of its outer face. Then we join $z_a^2$ to $x_b^1$ without edge crossing. Suppose $y_b^1$ is on the boundary of inner face $F$ of the embedding of $H_1(Y^1,Z^2)$, put the embedding of $H_1(Z^1,X^2)$ in face $F$ with $x_a^2$ on the boundary of its outer face, then we join $x_a^2$ to $y_b^1$ without edge crossing. After adding both $x_a^2y_b^1$ and $z_a^2x_b^1$ to $H_{1}^1$ without edge crossing, we get a planar graph $\widetilde{H}_{1}^1$. With the same process, we  add both $x_a^1y_b^2$ and $z_a^1z_b^2$ to $H_{1}^2$ without edge crossing, then we get a planar graph $\widetilde{H}_{1}^2$.
From Lemma \ref{le4}, we can also add $y_a^2z_b^1$ to $H_{2}^1$, and $y_a^1z_b^2$ to $H_{2}^2$ without edge crossing, then we get planar graphs
$\widetilde{H}_{2}^1$ and $\widetilde{H}_{2}^2$ respectively.

Then we get a planar decomposition $$\{\widetilde{H}_{1}^1, \widetilde{H}_{2}^1,H_{3}^1,\ldots, H_{p+1}^1, \widetilde{H}_{1}^2, \widetilde{H}_{2}^2, H_{3}^2,\ldots, H_{p+1}^2\}$$ of $K_{4p+3,4p+3,4p+3}\times K_2$ with $2p+2$ subgraphs.

Summarizing Cases $1$ and $2$, the lemma follows.
\end{proof}

\begin{theorem} The thickness of the Kronecker product of $K_{n,n,n}$ and $K_2$ is
$$\theta(K_{n,n,n}\times K_2)=\big\lceil\frac{n+1}{2}\big\rceil.$$ \end{theorem}

\begin{proof}
Because of $\ E(K_{n,n,n}\times K_2)=6n^2$ and  $\ V(K_{n,n,n}\times K_2)=6n$, from Theorem \ref{2.2}, we have
$$\theta(K_{n,n,n}\times K_2)\geq\big\lceil\frac{6n^2}{2(6n)-4}\big\rceil=\big\lceil{\frac{n}{2}+\frac{n}{6n-2}}\big\rceil=\big\lceil\frac{n+1}{2}\big\rceil.\eqno(3)$$

When $n=4p+2$, because $K_{4p+2,4p+2,4p+2}\times K_2$ is a subgraph of $K_{4p+3,4p+3,4p+3}\times K_2$, we have $\theta(K_{4p+2,4p+2,4p+2}\times K_2)\leq \theta(K_{4p+3,4p+3,4p+3}\times K_2).$ Combining this fact with Lemmas \ref{le3} ,\ref{le5} and \ref{le6}, we have
$$\theta(K_{n,n,n}\times K_2)\leq\big\lceil\frac{n+1}{2}\big\rceil.\eqno(4)$$

From inequalities $(3)$ and $(4)$,  the theorem is obtained.\end{proof}

\bibliography{bibfile}

\end{document}